\documentclass[11pt]{article}
\usepackage{amsmath,amssymb,amsbsy,amsfonts,amsthm,latexsym,
            amsopn,amstext,amsxtra,euscript,amscd,amsthm,url}
\usepackage{epsfig, subfigure, afterpage}

\numberwithin{equation}{section}  

\newtheorem{Thm}{Theorem}
\newtheorem{conj}{Conjecture}

\newtheorem{lem}{Lemma}

\newcommand{\p}{\partial}

\newtheorem{prop}{Proposition}
\newcommand{\C}{\mathbb{C}}

\newcommand{\Z}{\mathbb{Z}}
\newcommand{\Q}{\mathbb{Q}}

\DeclareMathAlphabet{\curly}{U}{rsfs}{m}{n}

\newcommand{\li}{{\rm li}}

\newcommand{\be}{\begin{equation}}
\newcommand{\ee}{\end{equation}}
\newcommand{\benn}{\begin{equation*}}
\newcommand{\eenn}{\end{equation*}}
\newcommand{\bal}{\begin{align*}}
\newcommand{\ea}{\end{align*}}
\newcommand{\eal}{\ensuremath{\end{align*}}}
\newcommand{\bea}{\begin{eqnarray}}
\newcommand{\eea}{\end{eqnarray}}

\newcommand{\lam}{\ensuremath{\lambda}}

\renewcommand{\a}{\ensuremath{\alpha}}

\newcommand{\del}{\ensuremath{\delta}}
\newcommand{\eps}{\ensuremath{\varepsilon}}

\renewcommand{\(}{\left(}
\renewcommand{\)}{\right)}
\newcommand{\ds}{\displaystyle}
\newcommand{\pfrac}[2]{\left(\frac{#1}{#2}\right)}

\newcommand{\CA}{\ensuremath{\mathcal{A}}}
\newcommand{\CE}{\ensuremath{\mathcal{E}}}

\renewcommand{\le}{\leqslant}

\renewcommand{\ge}{\geqslant}

\newcommand{\order}{\asymp}
\newcommand{\BB}{\mathcal{B}}
\renewcommand{\AA}{\mathcal{A}}
\newcommand{\ssum}[1]{\sum_{\substack{#1}}}  

\newcommand{\EKq}{\ensuremath{\gamma_q}}  

\makeatletter
\renewcommand{\pmod}[1]{\allowbreak\mkern7mu({\operator@font mod}\,\,#1)}
\makeatother

\begin{document}
\title{ ~~\\ Values of the Euler $\phi$-function not divisible by
a given odd prime, and the distribution of Euler-Kronecker constants for
cyclotomic fields}
\author{Kevin Ford, Florian Luca and Pieter Moree}
\date{}
\maketitle
{\def\thefootnote{}
}

\begin{abstract}
\noindent 
Let $\phi$ denote Euler's phi function. For a fixed odd
prime $q$ we investigate the first and second order terms of the
asymptotic series expansion for the
number of $n\le x$ such that $q\nmid \phi(n)$.
Part of the analysis involves a careful study of the Euler-Kronecker constants for
cyclotomic fields.  In particular, we show that the Hardy-Littlewood
conjecture about counts of prime $k$-tuples and a conjecture of Ihara 
about the distribution of these Euler-Kronecker constants cannot be both true.
\end{abstract}

\section{Introduction}
Let
$B(x)$ denotes the counting function of integers $n\le x$ that can be written as sum
of two squares.  In 1908, Landau \cite{L} 
proved the asymptotic formula
\begin{equation}\label{1908}
B(x)\sim \frac{Kx}{\sqrt{\log x}}
\end{equation}
for a certain positive constant $K$.  Landau's proof is based on the analytic theory of Dirichlet $L$-functions,
which come into play because a number $n$ is the sum of two squares if and only if each prime $p|n$ with
$p\equiv 3\pmod{4}$ divides $n$ with an even exponent.
The next year, Landau (\cite{La09a}; see also
\cite[\S 176--183]{La09b}) found a general asymptotic for the number of integers $n\le x$
which are divisible by no prime $p\in S$, where $S$ 
is any set of reduced residue classes modulo a fixed, but arbitrary, positive integer $q$.
In the case where $q$ is an odd prime and $S=\{1\pmod{q}\}$, let $\CA_q(x)$ be the counting function of such $n$.

Let $\phi$ denote Euler's phi function. For fixed odd prime $q$, let $$\CE_q(x)=|\{n\le x :q\nmid \phi(n)\}|.$$ 
Since $q\nmid \phi(n)$
if and only if $q^2\nmid n$ and $p\nmid n$ for all primes $p\equiv 1\pmod{q}$, it follows that
$\CE_q(x) = \CA_q(x) - \CA_q (x/q^2)$.  Landau's theorem immediately implies that
\be\label{uno}
\CE_q(x) \sim \frac{e_0(q) x}{(\log x)^{\frac{1}{q-1}}}
\ee
for some positive constant $e_0(q)$.

A standard application of the Selberg-Delange method (e.g., \cite[Theorem B]{Del}) 
gives an asymptotic expansion
\begin{equation}
\label{zozithet}
{\cal E}_q(x)= \frac{x}{(\log x)^{1/(q-1)}} \(e_0(q) + 
\frac{e_1(q)}{\log x}+\cdots+\frac{e_k(q)}{\log^k x}+O_k\pfrac{1}{\log^{k+1} x} \),
\end{equation}
with $e_j(q)$ being certain constants depending on $q$ and $k\ge 1$ an arbitrary natural number.
One of the main topics of this paper is the behavior of the second order terms $e_1(q)$
from \eqref{zozithet}. 

To place our problem in historical context, recall that Gauss's approximation $\li(x)=\int_2^x dt/\log t$ 
is a much better estimate
of $\pi(x)$, the number of primes up to $x$, than is $x/\log x$.  
Possibly inspired by this fact, 
in his first letter (16 Jan.\,1913) to Hardy, Ramanujan claimed that, for every $r\ge 1$,
\begin{equation}
\label{boehoe}
B(x)=K\int_2^{x} \frac{dt}{\sqrt{\log t}} +O \pfrac{x}{\log^r x}.
\end{equation}
However, Shanks \cite{Sh} showed that (\ref{boehoe}) is false for every $r>3/2$.
On the other hand he showed that the first term in (\ref{boehoe}) yields a closer approximation to
$B(x)$ than does $Kx/\sqrt{\log x}$.
 Similarly, in an unpublished manuscript, possibly included with his final letter (12 Jan.\,1920) to Hardy, 
Ramanujan discussed congruence properties of $\tau(n)$, the coefficient of $q^n$ in $q\prod_{k=1}^\infty (1-q^k)^{24}$, and $p(n)$, the partition function (see \cite{AB} or \cite{BO}).
For a finite set of special primes $q$ and positive constants $\delta_q$, 
Ramanujan claimed that ``it can be shown by transcendental methods that
\begin{equation}
\label{boehoe1}
\ssum{n\le x \\ q\nmid \tau(n)} 1 =C_q\int_2^x \frac{dt}{(\log t)^{\delta_q}}+O\pfrac{x}{\log^r x},
\end{equation}
where $r$ is an positive number''. Although asymptotically correct (as shown by Rankin and Rushforth),
the third author \cite{M}
showed that all claims of the form \eqref{boehoe1} are false for every $r>1+\delta_q$.

It is natural to ask which of the following two
approximations is asymptotically closer to $\CE_q(x)$, the \emph{Landau approximation}
\[
\mathcal{L}_q(x) = \frac{e_0(q)x}{\log^{1/(q-1)}x}
\]
or the \emph{Ramanujan approximation} 
\[
\mathcal{R}_q(x) = e_0(q)\int_2^x \frac{dt}{\log^{1/(q-1)}t}.\]
Integration by parts gives
\[
\mathcal{R}_q(x)=\frac{e_0(q)}{(\log x)^{\frac{1}{q-1}}} \( 1 + \frac{1}{(q-1)\log x} + 
O\pfrac{1}{\log^2 x}\),
\]
and it follows that if $(q-1)e_1(q)/e_0(q) > \frac{1}{2}$, then there exists $x_0$ such that
\begin{equation}
\label{vergelijk} 
|\CE_q(x)-\mathcal{R}_q(x)|<|\CE_q(x)-\mathcal{L}_q(x)|,~\forall x\ge x_0.
\end{equation}
If (\ref{vergelijk}) holds, we say that the Ramanujan approximation is asymptotically closer than the Landau approximation.

\begin{Thm} 
\label{maintwo}
Let $q$ be
an odd prime. For $q\le 67$ the Ramanujan approximation $\mathcal{R}_q(x)$ is asymptotically closer than the
Landau approximation $\mathcal{L}_q(x)$ for ${\cal E}_q(x)$, and for all remaining
primes the Landau 
approximation is asymptotically closer.  That is, $(q-1)e_1(q)/e_0(q) > \frac{1}{2}$
precisely when $q\le 67$.
\end{Thm}
Whereas before only a finite number of `Landau versus Ramanujan' comparison problems were settled, Theorem 
\ref{maintwo} extends this to an infinite number.  The following result reveals in fact that neither $\mathcal{L}_q(x)$ nor $\mathcal{R}_q(x)$ capture
the second term of the expansion \eqref{zozithet}. Throughout this paper, by ERH we mean that all nontrivial zeros of the Dirichlet $L$-functions
for characters modulo $q$ lie on the critical line $\Re s=\frac12$.
\begin{Thm}\label{e0e1}
 We have 
\[
 \frac{e_1(q)}{e_0(q)} = \frac{1-\gamma}{q-1} + 
\begin{cases}
O\pfrac{\log^2 q}{q^{3/2}} & \text{unconditionally with an effective constant}, \\
O_\eps\pfrac{1}{q^{2-\eps}} & \forall \eps>0, \text{ unconditionally with an ineffective constant} \\
O\pfrac{\log^2 q}{q^2} & \text{if there are no exceptional zeros for } q \\
O\pfrac{(\log q)\log\log q}{q^2} & \text{on ERH for $L$-functions modulo $q$}.
\end{cases}
\]
Here $\gamma=0.57721566\ldots$ is Euler's constant, and in this paper,
an exceptional zero is a real number $\beta > 1-1/(9.645908801\log q)$ that is a zero
of $L(s,\chi_q)$, with $\chi_q$ being the real, nonprincipal character modulo $q$.
\end{Thm}
McCurley \cite[Theorem 1.1]{Mc} showed that for each $q$, the region
$$\Re s \ge 1-\frac{1}{9.645908801 \log \max(q,q|\Im s|,10)}$$ contains at most one zero of 
$\prod_{\chi \mod q} L(s,\chi)$, and if the zero exists, it is real, simple and a zero of 
$L(s,\chi_q)$.

The remainder of the introduction is organized as follows.  Subsection \ref{e0} presents
the necessary analytic theory to understand $e_0(q)$.
In subsection \ref{e1}, 
we express the ratio $e_1(q)/e_0(q)$ in terms of two additional quantities $S(q)$ and $\gamma_q$,
(defined in \eqref{S(q)} and \eqref{EK}, respectively) and which are intesting to study in their own right.
We also state a theorem about the behavior of $S(q)$.  Subsection \ref{EK1} gives some general background
on $\gamma_q$ (called an Euler-Kronecker constant), and in subsection \ref{EK2} we present several theorems
and conjectures about $\gamma_q$.

A paper by Spearman and Williams \cite{SW} inspired us to study ${\cal E}_q(x)$.
In a rather roundabout way they obtained the asymptotic (\ref{uno}) (but not (\ref{zozithet})) and gave
an expression for $e_0(q)$ involving invariants of cyclotomic fields. We point out in
the next subsection that on using the Dedekind zeta function of a cyclotomic field, one
can rederive their expression (\ref{duo}) for $e_0(q)$ more directly.

\subsection{The first order term in \eqref{zozithet}}\label{e0}
The basis of \eqref{zozithet} is
an analysis of the Dirichlet series generating function for $n$ with
$q \nmid \phi(n)$, namely
\be\label{hdef}
h_q(s) = (1+q^{-s}) \prod_{\substack{p\ne q \\ p\not\equiv 1\pmod{q}}} (1-p^{-s})^{-1}
= (1-q^{-2s}) \zeta(s) \prod_{p\equiv 1\pmod{q}} (1-p^{-s}),
\ee
where $\zeta(s)$ is the Riemann zeta function.  Roughly speaking, the Selberg-Delange 
method provides an asymptotic expansion for $\sum_{n\le x} a_n$ in decreasing powers
 of $\log x$ provided that the Dirichlet series $\sum_{n=1}^\infty a_n n^s$ behaves
 like $\zeta(s)^z$ for some fixed complex number $z$.  If $a_n$ is multiplicative,
 this means that $a_p$ has average value $z$ over primes $p$.  In our case,
 $z=\frac{q-2}{q-1}$ by the prime number theorem for arithmetic progressions.
We observe that the primes $p\equiv 1\pmod{q}$ are precisely those primes which 
split completely in $K(q):=\mathbb Q(e^{2\pi i/q})$ and thus $\zeta_{K(q)}(s)$, 
the Dedekind zeta function  of $K(q)$, comes
into play.  We prove the following in  Section  \ref{AnTh}.

\begin{prop}\label{e0_fromzetaK}
Let $q$ be an odd prime.  Then
\begin{equation}
\label{fqs}
h_q(s)=\frac{(1-q^{-2s}) \zeta(s)}{\Big[C(q,s)(1-q^{-s})\zeta_{K(q)}(s)\Big]^{\frac{1}{q-1}}},
\end{equation}
where
\begin{equation}
\label{constant}
C(q,s)=\prod_{\substack{p\ne q\\ f_p\ge 2}} \Big(1-\frac{1}{p^{sf_p}}\Big)^{\frac{q-1}{f_p}},
\end{equation}
and $f_p=\operatorname{ord}_q p$ (the least positive $f$ with $p^f\equiv 1\pmod{q}$).
Furthermore,
$$
e_0(q)=\frac{1-q^{-2}}{\Gamma(\frac{q-2}{q-1})
\(C(q)(1-\frac{1}{q})\alpha^{\phantom{p}}_{K(q)}\)^{\frac{1}{q-1}}}, \qquad
\alpha^{\phantom{p}}_{K(q)} = \operatornamewithlimits{Res}_{s=1} \, \zeta_{K(q)}(s).
$$
\end{prop}

The main result in Spearman and Williams \cite{SW} is the asymptotic (\ref{uno}) with $e_0(q)$ expressed in terms of the
parameters of $K(q)$; namely, 
\begin{equation}\label{duo}
e_0(q)=\frac{(q+1)(q-1)^{\frac{q-2}{q-1}}\Gamma(\frac{1}{q-1}) \sin(\frac{\pi}{q-1})}
{2^{\frac{q-3}{2(q-1)}}q^{\frac{3(q-2)}{2(q-1)}}\pi^{\frac32}(h(K(q))R(K(q))C(q))^{\frac{1}{q-1}}},
\end{equation}
where $h(K(q))$ denotes the class number of $K(q)$ and $R(K(q))$ is its regulator. 
Spearman and Williams gave a rather involved description of $C(q)$, see
Section \ref{slechtec}.  Making use of the Euler product for $\zeta_{K(q)}(s)$, we will show
that actually $C(q)=C(q,1)$.
We have, for example,  $C(3)=\prod_{p\equiv 2 \pmod{3}}(1-1/p^2)$ (this is
Lemma 3.1 of \cite{SW}).  Our argument also gives a very short proof of an estimate of a product
from \cite{SW} (inequality  (\ref{prodigee}) below).  
On using that 
$\Gamma (\frac{1}{q-1}) \Gamma(\frac{q-2}{q-1}) = \frac{\pi}{\sin(\pi/(q-1))}$
and formula (\ref{residu}) for $\alpha^{\phantom{p}}_{K(q)}$ below, it is seen  that
the $e_0(q)$ as given in Proposition \ref{e0_fromzetaK} matches the formula
\eqref{duo}.

\subsection{The second order term in \eqref{zozithet}}\label{e1}
Our argument for Theorems \ref{maintwo} and \ref{e0e1}
proceed by first relating the $e_1(q)/e_0(q)$ to two additional quantities,
\be\label{S(q)}
S(q) = \frac{1}{q-1} \frac{C'(q,1)}{C(q,1)} = \sum_{p\ne q,f_p\ge 2}  
\frac{\log p} {p^{f_p}-1}
\ee
and the \emph{Euler-Kronecker constant}
\be\label{EK}
\gamma_q = \gamma_{K(q)}^{\phantom{p}} = \lim_{s\to 1} \( \frac{\zeta_{K(q)}(s)}{\alpha^{\phantom{p}}_{K(q)}} -\frac{1}{s-1} \).
\ee
\begin{prop}\label{TqL}
We have 
\[
(q-1)\frac{e_1(q)}{e_0(q)}
= 1 - \gamma + \frac{(3-q)\log q}{(q-1)^2(q+1)} + S(q) +
\frac{\gamma_q}{q-1}.
\]
\end{prop}

In Section \ref{Sec:S(q)}, we prove the following upper estimates for $S(q)$:

\begin{Thm}\label{thm:S(q)}$~$\\
{\rm (a)} We have $\ds S(q) \le 45/q$ for all $q$;\\
{\rm (b)} Let $\epsilon>0$ be fixed. The inequality 
$S(q)<\epsilon/q$ holds for $(1+o(1))\pi(x)$ primes $q\le x$.
\end{Thm}

The analysis used to prove Theorem \ref{thm:S(q)} depends on 
estimates for linear forms in logarithms to deal with the summands with
$p$ and $f_p$ both small.

As we will see, $\gamma_q$ is typically around $\log q$ and hence 
Theorem \ref{thm:S(q)} allows us to deduce that $\gamma_q$ has a larger influence
on the ratio $e_1(q)/e_0(q)$ than does $S(q)$.
The Euler-Kronecker constant (or invariant) can be defined for {\it any} number field.
Some history and basics will be recalled in the next section.


\subsection{Euler-Kronecker constants for number fields}\label{EK1}
For a general number field $K$ we have, for $\Re (s)>1$, the 
\emph{Dedekind zeta function}
$$
\zeta_K(s)=\sum_{\mathfrak{a}} \frac{1}{N{\mathfrak{a}}^{s}}
=\prod_{\mathfrak{p}}\frac{1}{1-N{\mathfrak{p}}^{-s}}.
$$
Here, $\mathfrak{a}$ runs over non-zero ideals in ${\cal O}_K$, the ring of integers of $K$,
$\mathfrak{p}$ runs over the prime ideals in ${\cal O}_K$ and $N{\mathfrak{a}}$ is the norm
of $\mathfrak{a}$.
It is known that $\zeta_K(s)$ can be analytically continued to $\C - \{1\}$,
and that at $s=1$ it has a simple pole with residue $\alpha^{\phantom{p}}_K$,
where \cite[Theorem 61]{FT}
\begin{equation}
\label{zoklassiek}
\alpha^{\phantom{p}}_K=\frac{2^{r_1}(2\pi)^{r_2}h(K)R(K)}{w(K)\sqrt{|d(K)|}},
\end{equation}
where $K$ has $r_1$ (resp. $2r_2$) real (resp. complex) embeddings, class number $h(K)$,
regulator $R(K)$, $w(K)$ roots of unity, and discriminant $d(K)$.
The Dedekind zeta function $\zeta_K(s)$ has a Laurent expansion
\begin{equation}
\label{laurent}
\zeta_K(s)=\frac{\alpha^{\phantom{p}}_K}{s-1}+c_0(K)+c_1(K)(s-1)+c_2(K)(s-1)^2+\cdots
\end{equation}
The ratio $\gamma_K^{\phantom{p}} = c_0(K)/\alpha_K^{\phantom{p}}$ is called the 
{\it Euler-Kronecker constant}
of $K$ (in particular $\gamma_{\Q}^{\phantom{p}}=\gamma$ is
Euler's constant).   This terminology originates
with Ihara \cite{I}.  
In the older literature (for references up to 1984 see, e.g., Deninger \cite{Deninger}) the focus
was on determining $c_0(K)$. As Tsfasman \cite{T} points out, $\gamma_K$ is of order $\log \sqrt{|d(K)|}$, whereas
$\alpha^{\phantom{p}}_K$ may happen to be exponential in $\log \sqrt{|d(K)|}$.

\indent The case where $K$ is quadratic has a long history. Since
$\zeta_{\mathbb{Q}(\sqrt{D})}(s)=\zeta(s)L(s,\chi_D)$, with $\chi_D=(D/n)$ the Kronecker symbol, we obtain
$$\gamma_{\mathbb{Q}(\sqrt{D})}=\gamma+\frac{L'(1,\chi_D)}{L(1,\chi_D)},$$
by partial differentation on using that $L(1,\chi_D)\ne 0$.
In the case when $K$ is imaginary quadratic the well-known
Kronecker limit formula expresses $\gamma_K^{\phantom{p}}$ in terms of special values of the Dedekind $\eta$--function
(see e.g. Section 2.2 in \cite{I}). An alternative expression involves a sum of logarithms of the Gamma function
at rational values.
Equating both expressions the Chowla-Selberg formula is obtained. Deninger \cite{Deninger} worked out the
analogue of the latter formula for real quadratic fields.\\
{\tt Numerical example}.
$$\gamma_{\mathbb{Q}(i)}=\gamma+\frac{L'(1,\chi_{-4})}{L(1,\chi_{-4})}
=\log\(\frac{\xi^2}{2} e^{2\gamma}\)\approx 0.82282525,{\rm
  ~where~}~\xi=\sqrt{\frac{2}{\pi}}\Gamma\Big({\frac34}\Big)^2.$$ 
(The number $\xi$ is also the arithmetic-geometric-mean (AGM) of $1$ and $\sqrt{2}$.)

Put
$$
{\tilde \zeta}_K(s)=s(1-s)\pfrac{\sqrt{|d(K)|}}{2^{r_2}\pi^{[K:\mathbb{Q}]/2}}^s 
\Gamma\({\frac{s}{2}}\)^{r_1}\Gamma(s)^{r_2}\zeta_K(s).$$
Then it is known that the functional equation ${\tilde \zeta}_K(s)={\tilde \zeta}_K(1-s)$ holds. Since
${\tilde \zeta}_K(s)$ is entire of order 1, one has the following Hadamard product factorization:
$${\tilde \zeta}_K(s)={\tilde \zeta}_K(0)e^{\beta_Ks}\prod_{\rho}\(1-\frac{s}{\rho}\)e^{s/\rho},$$
with some complex number $\beta_K$. Hashimoto et al.~\cite{veel} (cf.\,\,Ihara \cite[pp. 416--421]{I})
show that
$$-\beta_K=\sum_{\rho}\frac1{\rho}=\gamma_K^{\phantom{p}}-(r_1+r_2)\log 2+\frac12\log|d(K)|-\frac{[K:\mathbb Q]}{2}(\gamma+\log \pi)+1,$$
where the sum is over the zeros of $\zeta_K(s)$ in the critical strip. On specializing this to the
case $K(q)$, we obtain
\begin{equation}
\label{nul}
\sum_{\zeta_{K(q)}(\rho)=0}\frac{1}{\rho}=\EKq-(q-1)(\log 2 +\gamma)+
{\frac12}(q-2)\log q -\frac{(q-1)}{2}\log \pi.
\end{equation}
Since, at least conjecturally, $\EKq$ has normal order $\log q$ (see Theorem \ref{EH-EK} below),
this quantity seems to `measure' a subtle effect in the distribution of the zeros.

Prime ideals of small norm in the ring of integers of $K$ have a large influence on $\gamma_K^{\phantom{p}}$ 
as the following result (see, e.g., \cite{veel}) shows:
\begin{equation}
\label{ouddedoos}
\gamma_K^{\phantom{p}}=\lim_{x\rightarrow \infty}\Big(\log x -
\sum_{N\mathfrak{p}\le x} \frac{\log N\mathfrak{p}}{N\mathfrak{p}-1}\Big).
\end{equation}
As we shall see in the next subsection, in the special case $K=K(q)$, $\gamma_q$ is heavily influenced by
small primes which are congruent to $1$ modulo ${q}$.


\subsection{Euler-Kronecker constants for cyclotomic fields}\label{EK2}

In Section \ref{Sec:EKq} we study the distribution of $\EKq$ 
as $q$ runs through the primes.  In particular, we will give explicit estimates for
these constants needed for proving Theorems \ref{maintwo} and \ref{e0e1}.

In \cite{I}, Ihara remarks that it seems very likely that always $\EKq>0$ (this
was checked numerically for $q\le 8000$ by Mahoro Shimura, assuming
ERH).  Ihara
observed that $\gamma_K^{\phantom{p}}$ can be conspicuously negative and that this occurs when $K$ 
has many primes having
small norm (cf. (\ref{ouddedoos})). However, in the case of $K(q)$ the smallest norm is $q$ and therefore is rather
large as $q$ increases.

Using a new, fast algorithm (requiring computation of $L(1,\chi)$ for
all characters modulo $q$; see formula \eqref{EKL} below), we performed a search for
small values of $\EKq$.  The details of the algorithm and computation
are described later  in  Section \ref{Sec:EKq}.
One negative value was found, at  $q=964477901$.
We discuss later in the subsection the reason why this $q$, and conjecturally infinitely
many others, have negative Euler-Kronecker constants.

\begin{Thm}\label{EK<0}
For  $q=964477901$, we have
\[
\EKq = -0.18237\ldots
\]
\end{Thm}

In \cite{I}, Ihara also proved, under the assumption of ERH,
the one-sided bound $\EKq \le (2+o(1)) \log q$.
In \cite{I2},  Ihara made the following stronger conjecture.

\medskip

\noindent
{\bf Conjecture I (Ihara).}
{\it For any $\epsilon>0$, if $q$ is sufficiently large then
$$
\left(\frac12-\epsilon\right)\log q < \EKq<\left(\frac32+\epsilon\right)\log q.
$$ 
}

\smallskip

We will show, assuming the Hardy-Littlewood conjectures for counts of prime $k$-tuples, that
the lower bound in Ihara's conjecture is false and, even more, that
$\EKq$ is infinitely often negative. 
In 1904, Dickson \cite{Di} posed a wide generalization of the twin prime conjecture that is now
known as the ``prime $k$-tuples conjecture''.  It states that whenever a set of linear forms
$a_in+b_i$ ($1\le i\le k$, $a_i\ge 1$, $b_i\in\mathbb{Z}$) have no fixed prime factor (there is
no prime $p$ that divides $\prod_i (a_i n+b_i)$ for all $n$), then for infinitely  many 
$n$, all of the numbers $a_in+b_i$ are prime.  This expresses a kind of local-to-global
principle for prime values of linear forms, but is has not been proven for any $k$-tuple
of forms with $k\ge 2$. Later, Hardy and Littlewood \cite{HL} 
conjectured an asymptotic formula for the number of such $n$.
 There have been extensive numerical
studies of prime $k$-tuples, especially in the case $a_1=\cdots=a_k=1$,
providing evidence for these conjectures (e.g. \cite{Forbes, Forbesweb}).

In connection with 
$\EKq$, we need to understand special sets of forms.  We say that a
set $\{ a_1, \ldots, a_k\}$ of positive integers is an \emph{admissible set} if the collection of forms
$n$ and $a_in+1$ ($1\le i\le k$) have no fixed prime factor.  We need the following weak form of the 
Hardy-Littlewood conjecture:

\medskip

\noindent
{\bf Conjecture HL.} {\it Suppose $\AA=\{ a_1, \ldots, a_k\}$ is an admissible set.
The number of primes $n\le x$ for which the numbers $a_in+1$ are all prime is 
$\gg_\AA x(\log x)^{-k-1}$.}

\begin{Thm}\label{HL-EK}
Assume Conjecture HL. Then 
\[
\liminf_{q\to \infty} \frac{\EKq}{\log q} = -\infty.
\]
\end{Thm}

The basis of our theorem is the following formula for $\EKq$, cf. \eqref{ouddedoos}.

\begin{prop}\label{EKsump}
 We have
\begin{align*}
 \EKq &= - \frac{\log q}{q-1}  + 
\lim_{x\to\infty} \Bigg[ \log x - (q-1)\sum_{\substack{n\le x 
\\ n\equiv 1\pmod{q}}} \frac{\Lambda(n)}{n} \Bigg]\\
&=- \frac{\log q}{q-1} - (q-1)S(q) + 
\lim_{x\to\infty} \Bigg[ \log x - (q-1)\sum_{\substack{p\le x 
\\ p\equiv 1\pmod{q}}} \frac{\log p}{p-1} \Bigg].
\end{align*}
\end{prop}

The Euler-Kronecker constant $\gamma_q$ may also be easily expressed in terms of Dirichlet
$L$-functions at $s=1$; see \eqref{EKL} below in \S \ref{AnTh1}.

It is expected that the primes $p\equiv 1\pmod{q}$ behave very regularly for
$p>q^{1+\eps}$ (arbitrary fixed $\eps>0$).  It is irregularities in the
distribution of the $p\le q^{1+\eps}$ which provide the variation in the
values of $\EKq$.

Put $a(1)=0$ and inductively define $a(n)$ to be the smallest integer exceeding
$a(n-1)$ such that, for every prime $r$, the set $\{a(i)({\rm mod~}r):1\le i\le n\}$ has at
most $r-1$ elements (using the Chinese remainder theorem it is easily seen that the sequence
is infinite).
Given the prime $k$-tuples conjecture an equivalent statement is that
$a(n)$ is minimal such that there are infinitely many primes $q$ with $q+a(i)$ prime for
$1\le i\le n$. 
We have
$\{a(i)\}_{i=1}^{\infty}=\{0,2,6,8,12,18,20,26,30,32,\ldots\}$. This is sequence A135311
in the OEIS \cite{OEIS}
and is called `the greedy sequence of prime offsets'.
Given the prime $k$-tuples conjecture another equivalent statement is that $a(n)$
is minimal such that $a(1)=0$ and there are infinitely many primes $q$ with $a(i)q+1$ prime
for $2\le i\le n$, $n\ge 2$. 
Define $i_0$ to be the smallest integer satisfying
$$\sum_{i=2}^{i_0}\frac{1}{a(i)} >2,$$
A computer calculation gives $i_0=2089$ and $a(i_0)=18932$.

\begin{prop}\label{2089}
 Suppose that the number of primes $q$ such that $a(i)q+1$ is 
a prime for $2\le i\le 2089$ is $\gg x/(\log x)^{2090}$.  Then
$\EKq<0$ for $\gg x/(\log x)^{2090}$ primes $q\le x$.
\end{prop}

We note here that when $q=964477901$, then $aq+1$ is prime for 
$$a\in\{2,6,8,12,18,20,26,30,36,56,\ldots\}.$$

The strongest unconditional result about the distribution of primes in arithmetic 
progressions, the Bombieri-Vinogradov theorem, implies that the
primes $p\equiv 1\pmod{q}$ with $p>q^2(\log q)^A$ are well-distributed for
most $q$.  The Elliott-Halberstam conjecture \cite{EH} goes further:
Let $\pi(x;q,1)$ denote the number of primes $p\le x$ such that 
$p\equiv 1\pmod{q}$.
For convenience, write 
\[
E(q;x)=\pi(x;q,1) - \frac{\li(x)}{\phi(q)}.
\]

\smallskip
\noindent
{\bf Conjecture EH (Elliott-Halberstam).}
{\it 
For every $\eps>0$ and $A>0$,
\[
 \sum_{q\le x^{1-\eps}} | E(q;x) | \ll_{A,\eps} \frac{x}{(\log x)^A}.
\]
}
\smallskip

\begin{Thm}\label{EH-EK}
 (i) Assume Conjecture EH.  For every $\eps>0$, the bounds
\[
1 - \eps < \frac{\EKq}{\log q} < 1 + \eps
\]
hold for almost all primes $q$ (that is, the number of exceptional $q\le x$ is $o(\pi(x))$ as $x\to\infty$).

(ii) Assume Conjectures HL and EH.  Then the set $\{\EKq/\log q: q\text{ prime} \}$ is
dense in $(-\infty,1]$.
\end{Thm}

If, as widely believed, $E(x;q)$ is small for \emph{all} $q\le x^{1-\eps}$, we may
make a stronger conclusion.

\begin{conj}\label{EKq conj}
The set of limit points of  $\{\EKq/\log q: q\text{ prime} \}$ is $(-\infty,1]$.
\end{conj}

To determine the maximal order of $-\EKq$, one needs to assume a version of Conjecture HL
with the implied constant in the $\gg$-symbol explicitly depending on $\{a_1,\ldots,a_k\}$.
The heuristic argument in \cite[Proposition 5 and \S 9]{Gr} suggests that perhaps
\[
 \liminf \frac{\EKq}{(\log q)(\log\log\log q)} = -1.
\]

Our conditional results about $\EKq$ are proved using standard methods of analytic
number theory, and are very similar to the conditional bounds given by Granville in \cite{Gr}
for the class number ratio $h_q^-:=h(\mathbb{Q}(e^{2\pi i/q}))/h(\mathbb{Q}(\cos 2\pi/q))$.  
Kummer in 1851 conjectured that, as $q\rightarrow \infty$, one has
$$h_q^-\sim 2q\Big(\frac{q}{4\pi^2}\Big)^{(q-1)/4}.$$
This conjecture is the analog of the conjecture that $\EKq\sim \log q$ as $q\rightarrow \infty$.
We will make
use of several results from \cite{Gr}.

Our Theorem \ref{HL-EK} is reminiscent of a theorem of Hensley and Richards
\cite{HenRich}, who showed the incompatibility of the prime $k$-tuples conjecture
and a conjecture of Hardy and Littlewood about primes in short intervals.

Coming back to the connection between $\gamma_q$ and zeros of $\zeta_{K(q)}(s)$ (cf \eqref{nul}),
assuming ERH Ihara \cite{I2} defined
$$
c(q):=\Big(\sum_{\rho} \frac{q^{\rho-1/2}}{\rho(1-\rho)}\Big)/
\sum_{\rho} \frac{1}{\rho(1-\rho)}
=\Big(\sum_{\rho} \frac{\cos(\tau \log q)}{{\frac14}+\tau^2}\Big)/
\sum_{\rho} \frac{1}{\frac14+\tau^2},
$$
where $\rho=1/2+i\tau$ runs over all non-trivial zeros of $\zeta_{K(q)}(s)$. We have
$|c(q)|\le 1$ and
$$
\Big(\int_{\infty}^{\infty} \frac{\cos(t\log q)}{\frac14+t^2}dt\Big)/ \Big(\int_{-\infty}^{\infty}
\frac{dt}{\frac14+t^2}\Big)=\frac{1}{\sqrt{q}}.
$$
Thus, assuming that the distribution of $\tau$ modulo $2\pi/\log q$ for small $|\tau|$ is
rather uniform, we would maybe expect that $\sqrt{q}c(q)$ approximates 1 closely. 
Ihara \cite[Proposition 3]{I2} showed that under ERH we have
\begin{equation}
\label{iha}
\frac{\EKq}{\log q}=\frac32+(\sqrt{q}c(q)-1)+O\Big(\frac{1}{\log q}\Big).
\end{equation}
However, assuming
ERH and Conjecture HL, it follows from this and
Theorem \ref{HL-EK} that
$$\liminf_{q\rightarrow \infty}\sqrt{q}c(q)=-\infty.$$
Furthermore, assuming Conjecture EH, Theorem \ref{EH-EK} (i) and 
 (\ref{iha}) lead to the conjecture that the normal
order of $\sqrt{q}c(q)$ is $1/2$.


\subsection{The Euler-Kronecker constant for multiplicative sets}\label{EKS}
A set $S$ of positive integers is said to be {\it multiplicative} if for every pair $m$ and $n$ of coprime 
positive integers we have
that $mn$ is in $S$ iff both $m$ and $n$ are in $S$.  In other words, $S$ is a multiplicative set if
and only if the indicator function $f_S$ of $S$ is a multipicative function.

\noindent
{\tt Example 1}: the set of positive integers that can be written as a sum of two squares.\\
{\tt Example 2}: the set $S_q:=\{n\ge 1:~q\nmid \phi(n)\}$.\\
The Dirichlet series $L_S(s) := \sum_{n=1,~n\in S}^{\infty}n^{-s}$  converges for $\Re s > 1$. If
$L_S(s)$ has a simple pole at $s=1$ with residue $\delta>0$ and if
$$\gamma_S:=\lim_{s\rightarrow 1+0}\Big(\frac{L'_S(s)}{L_S(s)}+\frac{\delta}{s-1}\Big)$$
exists, we say that $S$ has Euler-Kronecker constant $\gamma_S$.
If we suppose that there exist $\delta,\rho>0$ such that
$$\sum_{p\le x,~p\in S}1=\delta\pi(x)+O\Big(\frac{x}{\log^{2+\rho}x}\Big),$$
then it can be shown that $\gamma_S$ exists.
In this terminology some of our results take a nicer form, e.g., in Theorem \ref{e0e1} we
now have
$\gamma_{S_q}=\gamma+O_\epsilon(q^{\epsilon-1})$ (with an ineffective constant). For
details and further results the reader is referred to Moree \cite{India}.

Finally, we like to point out that this paper is a very much reworked version of an earlier
preprint (2006) by the third author \cite{eerstev}. 
In it a proof of Theorem \ref{maintwo} on ERH is given. From the perspective
of computational number theory, this proof is far easier and less computation intensive 
than the one that does not assume ERH given here.


\section{Analytic Theory}\label{AnTh}


\subsection{Propositions \ref{e0_fromzetaK} and \ref{TqL}}\label{AnTh1}

We recall some facts from the theory of cyclotomic fields needed for our proofs. For 
a nice introduction to cyclotomic fields, see \cite{W}.
The following result, see
e.g. \cite[Theorem 4.16]{N}, describes the splitting of primes in the ring of integers of
a cyclotomic field.

\begin{lem} 
\label{washington}
(cyclotomic reciprocity law). Let $K=\mathbb Q(e^{2\pi i/m})$.
If the prime $p$ does not divide
$m$ and $f=\operatorname{ord}_m p$, then the principal 
ideal $p{\mathcal O}_K=\mathfrak{p}_1\cdots \mathfrak{p}_g$ with
$g=\phi(m)/f$, and all $\mathfrak{p_i}$'s distinct and of degree $f$.

However, if $p$ divides $m$, $m=p^am_1$ with $p\nmid m_1$ and $f=\operatorname{ord}_{m_1} p$, 
then $p{\mathcal O}_k=(\mathfrak{p}_1\cdots \mathfrak{p}_g)^e$ with $e=\phi(p^a)$, 
$g=\phi(m_1)/f$, and all $\mathfrak{p_i}$'s distinct and of degree $f$.
\end{lem}

In case $K=K(q)$, we have $r_1=0$, $2r_2=q-1$, $w(K)=2q$ (as $K$ contains
exactly $\{\pm 1,\pm \omega, \pm \omega^2, \ldots, \pm \omega^{q-1}\}$ as roots of unity, with
$\omega=e^{2\pi i/(q-1)}$)
and furthermore $D(K)=(-1)^{q(q-1)/2}q^{q-2}$, and thus we obtain from (\ref{zoklassiek}) that
 \begin{equation}
\label{residu}
\alpha^{\phantom{p}}_{K(q)}={\rm Res}_{s=1}\zeta_{K}(s)=2^{\frac{q-3}2}q^{-{\frac{q}2}}\pi^{\frac{q-1}2}h(K)R(K).
\end{equation}

For cyclotomic fields $K$ the Euler product for $\zeta_{K}(s)$ can be written down explicitly using
the ``cyclotomic reciprocity law''. We find that
\be\label{rezzie}
\begin{split}
\zeta_{K(q)}(s)&=\left(1-\frac{1}{q^s}\right)^{-1}\prod_{p\ne q}
\Big(1-\frac{1}{p^{sf_p}}\Big)^{\frac{1-q}{f_p}}\\
&= \left(1-\frac{1}{q^s}\right)^{-1}C(q,s)^{-1}\prod_{p\equiv 1({\rm mod~}q)}
\Big(1-\frac{1}{p^s}\Big)^{1-q}.
\end{split}
\ee

It is also well-known (see, e.g., \cite[Theorem 65]{FT}) that
\begin{equation}
\label{product}
\zeta_{K(q)}(s)=\zeta(s) \prod_{\chi\ne \chi_0} L(s,\chi),
\end{equation}
where the product is over characters $\chi$ modulo $q$, with $\chi_0$ being the 
principal character, and $L(s,\chi)$ the Dirichlet $L$-function associated with
 $\chi$.

\begin{proof}[Proof of Proposition \ref{e0_fromzetaK}]
First, \eqref{fqs} follows by combining \eqref{hdef} and \eqref{rezzie}.
By \eqref{product},
\be\label{hqLschi}
 h_q(s) =
(1+q^{-s}) \( \zeta(s) (1-q^{-s}) \)^{\frac{q-2}{q-1}}C(q,s)^{-\frac{1}{q-1}} \prod_{\chi\ne \chi_0}
L(s,\chi)^{-\frac{1}{q-1}}.
\ee
For $\chi\ne\chi_0$,  $L(s,\chi)$ is analytic and nonzero at $s=1$.
Hence, $f(s)=h_q(s) (s-1)^{(q-2)/(q-1)} s^{-1}$ is analytic in a 
neighborhood of $s=1$ and has a power series expansion there.
Moreover, $\prod_{\chi\ne \chi_0} L(s,\chi)$ has no zeros in the region
$\Re s \ge 1 - a_q (\log (|\Im s|+2))^{-1}$ for some positive $a_q$.  
Therefore, $h_q(s)/s$ has an expansion around the point $s=1$ of the form
\[
\frac{h_q(s)}{s}=\frac{1}{(s-1)^{(q-2)/(q-1)}}\Big(c_0(q)+c_1(q)(s-1)+\cdots+c_k(q)(s-1)^k+\cdots \Big),
\]
To apply the Selberg-Delange method, we also need a mild growth condition on
$h_q(s) \zeta(s)^{-\frac{q-2}{q-1}}$.
The function $C(q,s)$ is analytic for $\Re s > \frac12$, and uniformly bounded
in the half-plane $\Re s \ge \frac34$.  For $\sigma \ge 1 - a_q(2\log(|t|+2))^{-1}$,
\[
\Big| \prod_{\chi\ne \chi_0} L(\sigma+it) \Big|^{-1} \ll_q (\log (|t|+2))^{q-2}.
\]
By \cite[\S II.5, Theorem 3]{Ten},
an asymptotic expansion \eqref{zozithet} holds with the coefficients satisfying
$e_j(q) = c_j(q)/\Gamma(\frac{q-2}{q-1}-j)$.
\end{proof}

\begin{proof}[Proof of Proposition \ref{TqL}]
 By Proposition \ref{e0_fromzetaK}
and the functional equation $\Gamma(z)=(z-1)\Gamma(z-1)$, we have
\begin{align*}
 \frac{e_1(q)}{e_0(q)} &= -\frac{1}{q-1} \frac{c_1(q)}{c_0(q)} = -\frac{f'(1)}{(q-1)f(1)} \\
&= \frac{1}{q-1} \( 1 - \lim_{s\to 1^+} \(\frac{1-\frac{1}{q-1}}{s-1} + \frac{h_q'(s)}{h_q(s)} \)\).
\end{align*}
By the Laurent expansion $\zeta(s)=(s-1)^{-1}+\gamma+O(|s-1|)$, we have
\be\label{laurent_zeta}
 \frac{\zeta'(s)}{\zeta(s)} = -\frac{1}{s-1} + \gamma + O(|s-1|) \qquad (|s-1|\le 1).
\ee
Hence, by logarithmic differentiation of \eqref{hqLschi},
\[
\lim_{s\to 1^+} \frac{1-\frac{1}{q-1}}{s-1} + \frac{h_q'(s)}{h_q(s)}  = -\frac{\log q}{q+1}
+ \frac{(q-2)\log q}{(q-1)^2} + \frac{q-2}{q-1}\gamma - S(q) - \frac{1}{q-1}\sum_{\chi\ne \chi_0}
\frac{L'(1,\chi)}{L(1,\chi)}.
\]
By \eqref{laurent}, \eqref{laurent_zeta} and logarithmic differentiation of \eqref{product}, we have 
\be\label{EKL}
\EKq = \gamma + \sum_{\chi\ne \chi_0}
\frac{L'(1,\chi)}{L(1,\chi)}. 
\ee
On combining the various formulas the proof is completed.
\end{proof}


\subsection{The constant $C(q)$}\label{slechtec}
Spearman and Williams put, for a generator $\chi_q$ of the group of characters
modulo $q$,
\be\label{proddie}
C(q)=\prod_{r=1}^{q-2}
 \prod_{\chi_g(p)=\omega^r}\Big(1-\frac{1}{p^{(q-1)/(r,q-1)}}\Big)^{(r,q-1)}.
\ee
{From} this definition it is not a priori clear that $C(q)$ is independent of the choice
of $\chi_g$.  However, Spearman and Williams show that this is so.
\begin{prop}
\label{cc}
We have $C(q)=C(q,1)$.
\end{prop}

\begin{proof}
We claim that if $\chi_g(p)=\omega^r$, then $f_p=(q-1)/(r,q-1)$. We have
$1=\chi_g(p^{f_p})=\omega^{rf_p}$. It follows that $(q-1)|rf_p$ and thus $q_r=(q-1)/(r,q-1)$ must
be a divisor of $f_p$. On the other hand, since $\chi_g(a)=1$ if and only if $a=1$,
 it follows
from $\omega^{rq_r}=\chi_g(p^{q_r})=1$ and $q_r|f_p$, that $f_p=q_r$. Thus, we can rewrite
(\ref{proddie}) as
\begin{equation}
\label{proddie2}
C(q)=\prod_{r=1}^{q-2}\prod_{\chi_g(p)=\omega^r}\Big(1-\frac{1}{p^{f_p}}\Big)^{q-\frac{1}{f_p}}.
\end{equation}
Note that $p\ne q$ and $f_p\ge 2$ iff $\chi_g(p)=\omega^r$ for some $1\le r\le q-2$.
This observation in combination with the absolute convergence of the
double product (\ref{proddie2}), then shows that $C(q)=C(q,1)$. 
\end{proof}

\noindent {\bf Remark}. Proposition \ref{cc} says that $1/C(q)$ is the contribution at $s=1$
of the primes $p\ne q$, $p\not\equiv 1({\rm mod~}q)$ to the Euler product (\ref{rezzie}) of $K(q)$.


\subsection{On Mertens' theorem for arithmetic progressions}
\noindent A crucial ingredient in the paper of Spearman and Williams is the asymptotic estimate  \cite[Proposition 6.3]{SW}
that
\begin{equation}
\label{prodigee}
\prod_{\substack{p\le x\\ p\equiv 1({\rm mod~}q)}} \left(1-{\frac1{p}}\right)
=\Big({\frac{qe^{-\gamma}}{(q-1)\alpha^{\phantom{p}}_{K(q)}C(q)\log x}}\Big)^{\frac{1}{q-1}}
\Big(1+O_q\Big({\frac1{\log x}}\Big)\Big).
\end{equation}
An alternative, much shorter proof of the estimate (\ref{prodigee}) can be obtained
on invoking Mertens' theorem for algebraic number fields.
\begin{lem}
\label{rosie}
Let $\alpha^{\phantom{p}}_K$ denote the residue of $\zeta_K(s)$ at $s=1$. Then,
$$
\prod_{N\mathfrak p\le x}\Big(1-{\frac{1}{N\mathfrak p}}\Big)
={\frac{e^{-\gamma}}{\alpha^{\phantom{p}}_K\log x}}
\Big(1+O_K\left({\frac1{\log x}}\right)\Big),
$$
where the product is over the prime ideals $\mathfrak p$ in ${\cal O}_K$ having norm $\le x$.
\end{lem}
\begin{proof}
Similar to that of the usual Mertens' theorem (see e.g.\,Rosen \cite{R} 
or Lebacque \cite{Leba}).
\end{proof}
\noindent {\it Proof of estimate} (\ref{prodigee}). We invoke Lemma \ref{rosie} with $K=K(q)$ 
and work out the product over the prime ideals more explicitly using the cyclotomic reciprocity law, Lemma 
\ref{washington}. One finds, for $x\ge q$, that it equals
$$
\left(1-\frac{1}{q}\right)\prod_{\substack{p\le x\\p\equiv 1({\rm
      mod~}q)}}\Big(1-{\frac{1}{p}}\Big)^{q-1}
\prod_{\substack{p^{f_p}\le x,~p\neq q\\ p\not \equiv 1({\rm mod~}q)}}\Big(1-{\frac{1}{p^{f_p}}}\Big)^
{\frac{q-1}{f_p}}=$$
$$
\left(1+O_q\left(\frac{1}{\sqrt{x}}\right)\right)\left(1-{\frac{1}{q}}\right)C(q)
\prod_{\substack{p\le x\\ p\equiv 1({\rm mod~}q)}}\Big(1-{\frac1{p}}\Big)^{q-1},
$$
where we used that for $k\ge 2$,
$$
\prod_{p^k>x}(1-p^{-k})^{-1}=1+O(\sum_{n^k>x}{n^{-k}})=1+O(x^{1/k-1}).
$$ 
Thus, on invoking Lemma \ref{rosie}, we deduce \eqref{prodigee}.\qed\\

For recent work on this theme, the reader is referred to the papers by Languasco and Zaccagnini 
\cite{LZ, LZ2, LZ3, LZ4}.

%
%
\section{Estimates for the Euler-Kronecker constants $\EKq$}\label{Sec:EKq}
%
%


\subsection{Unconditional bounds for $\EKq$}


\begin{proof}[Proof of Proposition \ref{EKsump}]
Apply \eqref{EKL}, the orthogonality of characters, and the relation 
(e.g. \cite[\S 55]{La09b} or \cite[\S 6.2, Exercise 4]{MV})
\[
 \sum_{n\le x} \frac{\Lambda(n)}{n} = \log x - \gamma + o(1) \qquad (x\to \infty)
\]
to obtain the first claimed bound.  The sum on $n$ equals
\[
 \sum_{\substack{p\le x \\ p\equiv 1\pmod{q}}} \frac{\log p}{p-1} - A(x)+B(x),
\]
where
\[
 A(x)=\sum_{\substack{p\le x, p^a>x \\ p\equiv 1\pmod{q}}} \frac{\log p}{p^a},
\qquad B(x) = \sum_{\substack{p^a\le x \\ p^a\equiv 1\pmod{q} \\ p\not\equiv 1\pmod{q}}}
\frac{\log p}{p^a}.
\]
Clearly, $\lim_{x\to \infty} B(x) = S(q)$.  The last estimate we need is $\lim_{x\to\infty} 
A(x)=0$, which is proved as follows:
\[
 A(x) \le \sum_{a=2}^\infty \sum_{n > x^{1/a}} \frac{\log n}{n^a} \ll
\sum_{a=2}^\infty \frac{\log x}{a^2 x^{1-1/a}} \ll \frac{\log x}{\sqrt{x}}.\qedhere
\]
\end{proof}

\noindent {\bf Remark 1}. Alternatively one can prove Proposition \ref{EKsump} on making
the limit formula (\ref{ouddedoos}) explicit for $K(q)$ using Lemma \ref{washington}.

\noindent {\bf Remark 2}.  Proposition \ref{EKsump} can be used to approximate, nonrigorously, the
value of $\EKq$.  For example, when $q=964477901$, the right side in 
Proposition \ref{EKsump} stays very close to $-0.18$ for $10^6 \le x/q\le 10^7$;
see Theorem \ref{EK<0}.
\medskip

\begin{prop}\label{sumsmallp}
 If $y\ge 10q$ and $q\ge 11$, then
\[
\sum_{\substack{p\le y \\ p\equiv 1\pmod{q}}} \frac{\log p}{p-1}
\le \frac{2\log y + 2(\log q)\log\log (y/q)}{q-1}.
\]
\end{prop}

\begin{proof}
By the Montgomery-Vaughan sharpening of the Brun-Titchmarsh inequality \cite{MV73}, we have
\[
 \pi(y;q,1) \le \frac{2y}{(q-1)\log(y/q)},
\]
and hence, by partial summation,
\begin{align*}
\sum_{\substack{p\le y \\ p\equiv 1\pmod{q}}} &\frac{\log p}{p-1} =
 \frac{\pi(y;q,1)\log y}{y-1} + \int_{2q}^y \frac{\pi(t;q,1)}{(t-1)^2}\(\log t - \frac{t-1}{t}\)\, dt\\
&\le \frac{2}{q-1} \( \frac{y\log y}{(y-1)\log(y/q)} +
 \int_{2q}^y \frac{t}{(t-1)^2} + \frac{t\log q-(t-1)}{(t-1)^2\log(t/q)} dt  \)\\
&\le  \frac{2}{q-1} \( \frac{y}{y-1}\( 1 + \frac{\log q}{\log 10}\) + \int_{2q}^y
\frac{1}{t}+\frac{2}{(t-1)^2}+\frac{\log q}{t\log(t/q)} \, dt \) \\
&\le \frac{2}{q-1} \( 1.01+0.44\log q +\log (\frac{y}{2q}) + \frac{2}{2q-1}
+ (\log q)(\log\log (\frac{y}{q}) - \log\log 2)  \) \\
&\le \frac{2\log y+2(\log q)\log\log (y/q)}{q-1}.\qedhere
\end{align*}
\end{proof}

\begin{prop}\label{sieve}
 Uniformly for $z\ge 2$, $\delta>0$ and $0<\eps\le 1$, the number of primes $q\le z$ for
which
\[
 \sum_{\substack{p\le q^{1+\eps} \\ p\equiv 1\pmod{q}}} \frac{\log p}{p-1} \ge
\delta \frac{\log q}{q}
\]
is $O(\eps \pi(z)/\delta)$.
\end{prop}

\begin{proof}
By sieve methods  (e.g. \cite[Theorem 5.7]{HR}), for an even $k\ge 2$,
 the number of prime $q\le z$ with $kq+1$ prime 
is $O(\frac{k}{\phi(k)} \frac{z}{\log^2 z})$ uniformly in $k$.  Thus, the
number of primes $q$ in question is
\begin{align*}
&\le \sum_{q\le z} \frac{q}{\delta \log q} \sum_{\substack{p\le q^{1+\eps} \\ p\equiv 1\pmod{q}}}
   \frac{\log p}{p-1} 
\le {\frac{1}{\delta}} \sum_{\substack{k\le z^{\eps} \\ 2|k }} 
  \sum_{\substack{k^{1/\eps} \le q\le z\\ kq+1 \text{ prime}}} \frac{\log(kq+1)}{k\log q} \\
&\ll \frac{z}{\delta \log^2 z} \sum_{k\le  z^{\eps}} \frac{1}{\phi(k)} \ll \frac{\eps}{\delta}
\frac{z}{\log z},
\end{align*}
where we used the well-known estimate $\sum_{n\le x}\phi(n)^{-1}=O(\log x)$.
\end{proof}


\begin{lem}\label{beta}
 Let $q\ge 10000$ be prime and let $\chi$ be the quadratic character modulo $q$.
If $L(\beta_0,\chi)=0$, then 
\[
\beta_0 \ge 1 - \frac{3.125\min(2\pi,\frac12\log q)}{\sqrt{q}\log^2 q}. 
\]
\end{lem}

\begin{proof}
By Dirichlet's class  number formula \cite[\S 6, (15) and (16)]{Da},
\[
 L(1,\chi) =\displaystyle{ \begin{cases} \frac{\pi h(-q)}{\sqrt{q}} & q\equiv 3\pmod{4} \\ \\
              \frac{h(q)\log u}{\sqrt{q}} & q\equiv 1\pmod{4},
             \end{cases}}
\]
where $h(d)$ is the class number of $\mathbb{Q}(\sqrt{d})$, and $u$ is
the smallest unit in $\mathbb{Q}(\sqrt{d})$ satisfying $u>1$.  Since
$u > \sqrt{q}$ and $h(-s)\ge 2$ for $s>163$, we obtain for $q>163$ that
$L(1,\chi) \ge \min(2\pi,\frac12 \log q) q^{-1/2}$.  
Assume $\beta_0 \ge 1-0.2 q^{-1/2}$,
else there is nothing to prove.  Let $V(t)=\sum_{n\le t} \chi(n)$.  By the P\'olya-Vinogradov
inequality (\cite[\S 23, (2)]{Da} or \cite[\S 9.4]{MV}), for $t>u>0$,
\begin{align*}
 |V(t)-V(u)| &< \frac{2}{\sqrt{q}} \sum_{a=1}^{(q-1)/2} \frac{1}{\sin(\pi a/q)}
\le \frac{2}{\sqrt{q}} \int_{1/2}^{q/2} \frac{dt}{\sin(\pi t/q)} \\
&= \frac{2\sqrt{q}}{\pi} \log \cot \pfrac{\pi}{4q} < \frac{2}{\pi} \sqrt{q} \log
(4q/\pi).
\end{align*}
Hence, for $\frac12 \le \sigma \le 1$ and $y\ge 100$,
\begin{align*}
| L'(\sigma,\chi)| &\le y^{1-\sigma} \sum_{n\le y} \frac{\log n}{n}
+ \int_y^\infty |V(t)-V(y)| \frac{\sigma\log t-1}{t^{1+\sigma}}\, dt \\
&\le y^{1-\sigma}\( \frac{\log^2 y}{2} + \frac{2}{\pi}\sqrt{q} \log (\frac{4q}{\pi})
 \frac{\log y}{y} \).
\end{align*}
Taking $y=q^{0.67}$ gives
\[
 |L'(\sigma,\chi)| \le q^{0.67(1-\sigma)} ( 0.316\log^2 q) \le 0.32 \log^2 q.
\]
The mean value theorem implies $(1-\beta_0)(0.32\log^2 q) \ge L(1,\chi)$ and the lemma follows.
\end{proof}


\subsection{Numerical calculation of $\EKq$}


The identity \eqref{EKL} is useful for numerically calculating $\EKq$
 for small $q$.  For example, cf. \cite{M2},
$$
\gamma_3=\gamma+{\frac{L'(1,\chi_{3})}{L(1,\chi_{3})}}
=0.945497280871680703239749994158189073\ldots,
$$
where $\chi_{3}$ stands for the only non-principal character modulo $3$.
For larger $q$ we use the following formulas.  First,
\be\label{L1chi}
L(1,\chi) = -\frac{1}{q} \sum_{r=1}^{q-1} \chi(r) \psi\pfrac{r}{q}, \qquad
\psi(z) = \frac{\Gamma'(z)}{\Gamma(z)}.
\ee
We also use
\be\label{L'1chi}
-L'(1,\chi) = \sum_{n=1}^\infty \frac{\chi(n)\log n}{n} = 
(\log q)L(1,\chi) + \frac{1}{q}\sum_{r=1}^{q-1} \chi(r) T\pfrac{r}{q},
\ee
where
\[
T(y) = \sum_{m=0}^\infty \( \frac{\log(m+y)}{m+y} - \frac{\log (m+1)}{m+1}\).
\]
Here, the term $(m+1)^{-1}\log (m+1)$ is a convergence factor, included so that
the terms in the sum on $m$ are $O(m^{-2}\log m )$.  The advantage of using
\eqref{L1chi} and \eqref{L'1chi} is that for each $q$, there are only $q-1$ values of
$\psi$ and $q-1$ sums $T(r/q)$ to compute.  With these values in hand, there are, however,
still  $\gg q^2$ operations (additions, subtractions, multiplications, divisions) 
needed using a naive algorithm to compute all of the
numbers $L(1,\chi)$ and $L'(1,\chi)$. 
A significant speed-up is achieved by observing
that the vector of sums on $r$ on the right sides of \eqref{L1chi} and \eqref{L'1chi} are
discrete Fourier transform coefficients.  Specifically, let $g$ be a primitive root of $q$,
$\chi_1$ the character with $\chi_1(g)=e^{2\pi i/(q-1)}$ and for $1\le k\le q-1$, let $r_k$ be the integer in 
$[1,q-1]$ satisfying $g^k\equiv r_k\pmod{q}$.  The characters modulo $q$
are $\chi_0,\chi_1,\chi_1^2,\ldots,\chi_1^{q-2}$ and for $\chi=\chi_1^j$, the sum in \eqref{L1chi} is
$\sum_{k=1}^{q-1} e^{2\pi ijk/(q-1)} \psi(r_k/q)$ and the sum on $r$ in \eqref{L'1chi}
is $\sum_{k=1}^{q-1} e^{2\pi ijk/(q-1)} T(r_k/q)$.
Fast Fourier Transform (FFT) algorithms may be used to recover $L(1,\chi)$ and $L'(1,\chi)$ from
the vectors $(\psi(r_1/q),\ldots,\psi(r_{q-1}/q))$ and 
$(T(r_1/q),\ldots,T(r_{q-1}/q))$, respectively, with $O(q\log q)$ operations.

A program to compute the numbers $L(1,\chi)$ and $L'(1,\chi)$ was written in the C language, making use of
the  FFT library \texttt{fftw} \cite{fftw}.  Running on a Dell Inspiron 530 desktop computer with
Ubuntu Linux, 2GB RAM and a 2.0 GHz processor, the program computed $\EKq$ for all prime
$q\le 30000$ in 2 minutes.  All computations were performed using high precision arithmetic
(80-bit ``long double precision'' floating point numbers).
In order to handle very large $q$ (larger than about $5\times 10^7$)
a machine with more memory was required.  A suitably modified version
of the program was run on a large cluster computer, with 256GB RAM,
48 core AMD Opteron 6176 SE processors  (4 sockets, 12 cores/socket),
operating system  Ubuntu Linux 10.04.3 LTS x86\_64.  The computation of 
$\EKq$ for $q=964477901$ took 64 minutes of CPU time on this system.
This gave Theorem \ref{EK<0}.

\begin{lem}\label{EKq30000}
 For $q\le 30000$, we have $0.315 \log q \le \EKq\le 1.627 \log q$.
\end{lem}

The largest value of $\EKq/\log q$ among $q\le 30000$
is $\gamma_{19}/\log 19 = 1.626\ldots$ and the smallest is 
 $\gamma_{17183}/\log 17183 = 0.315\ldots$.  Lemma \ref{EKq30000} suffices for 
the application to Theorem \ref{maintwo}.

In the next subsection, we will discuss more about the likely distribution of the Euler-Kronecker
constants.  Figure 1 displays a scatter plot of  $\EKq/\log q$ for the
primes $q\le 50000$.  

\begin{figure}[t]\label{EKq50000}
\epsfig{file=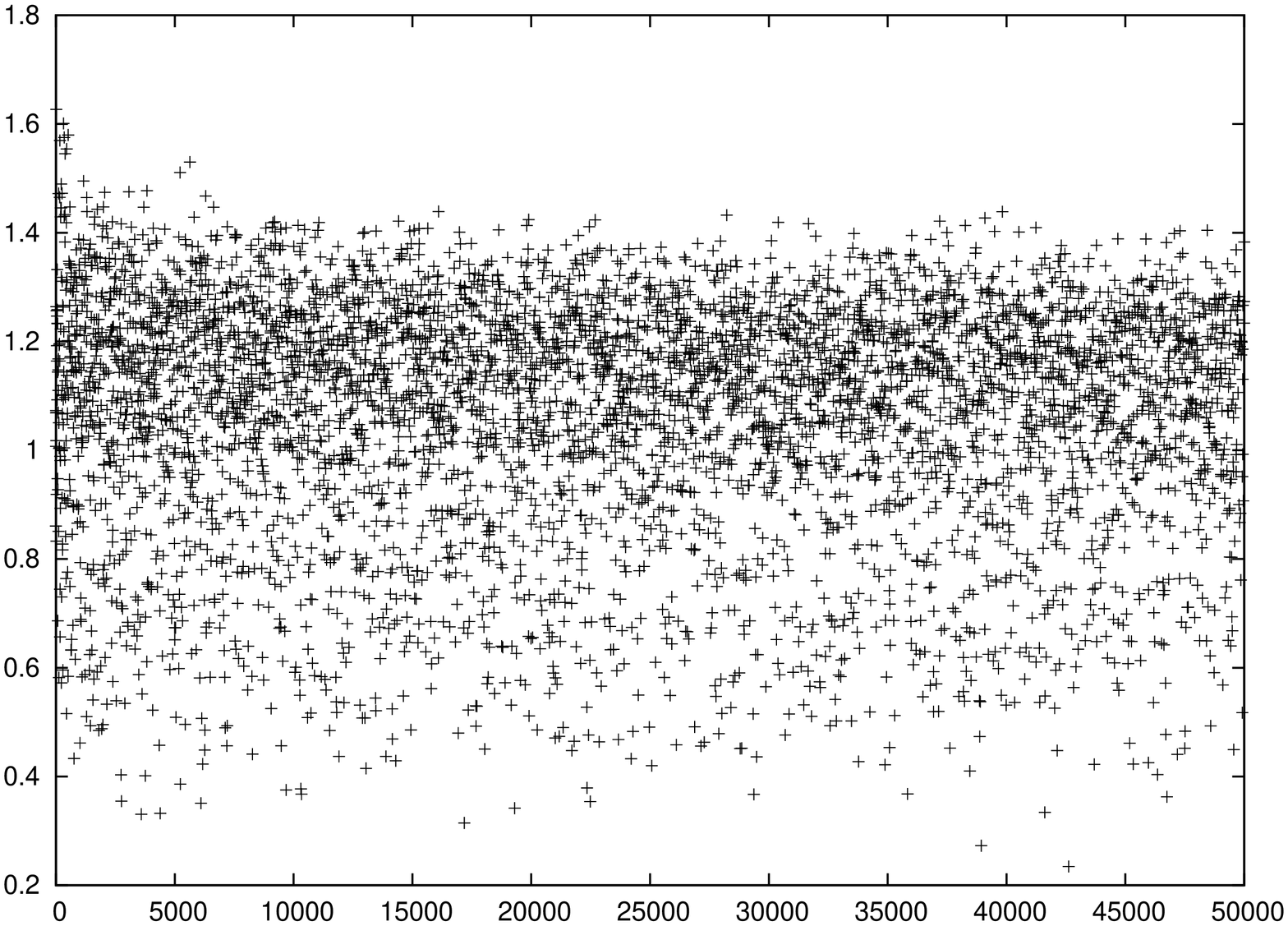, height=5.4in, width=3in, angle=270}
\bigskip
\medskip
\caption{$\EKq/\log q$ for $q\le 50000$}
\end{figure}

%
%
\subsection{Conditional bounds for $\EKq$}
%
%

\begin{lem}\label{xq2}
(i) For all $C>0$ and for all except $O(\pi(u)/(\log u)^C)$ primes $q\le u$,
\[
\EKq = 2\log q - q \sum_{\substack{p\le q^2
\\ p\equiv 1\pmod{q}}} \frac{\log p}{p-1} + O_C(\log\log q).
\]
(ii) Assuming ERH, the above inequality holds for all prime $q$ (the implied constant in the 
$O_C(\log\log q)$ term being absolute in this case).

(iii) Assume Conjecture EH and fix $C>0$ and $\eps>0$.  For all except 
$O(\pi(u)/(\log u)^C)$ primes $q\le u$,
\[
\EKq = (1+\eps)\log q - q \sum_{\substack{p\le q^{1+\eps}
\\ p\equiv 1\pmod{q}}} \frac{\log p}{p-1} + O_C(\log\log q).
\]
\end{lem}

\begin{proof}
Part (i) follows by a straightforward combination of Proposition \ref{EKsump} and the Bombieri-Vinogradov
theorem \cite[\S 28]{Da} (cf. Proposition 2 of \cite{Gr}).  
The latter states that for all $A>0$ there is a $B$ so that
\[
 \sum_{q\le \sqrt{x}/\log^B x} \left| E(x;q) \right|
\ll \frac{x}{(\log x)^A}.
\]
For any $x\ge z > q$, partial summation implies
\be\label{range}
\sum_{\substack{y\le p\le x \\ p\equiv 1\pmod{q}}} \frac{\log p}{p-1} = 
\frac{\log\pfrac{x-1}{y-1}}{q-1}+G_q(x,z),
\ee
where
\[
 G_q(x,z) = \left[ \frac{E(x;t)\log t}{t-1} \right]_{y}^x + \int_y^x \( \frac{\log t}{(t-1)^2}
-\frac{1}{t^2-t}\) E(t;q) \, dt.
\]
Let $B$ be the constant corresponding to $A=C+3$, let $z$ be large and put $y=z^2(\log z)^{2B+1}$.
For any $t\ge y$, $2z \le \sqrt{t}(\log t)^{-B}$ and so
\[
 S(t;z) :=\sum_{z<q\le 2z} |E(t;q)| \ll \frac{t}{(\log t)^{C+3}}.
\]
We obtain
\[
 \sum_{z<q\le 2z} \sup_{x>y}\left| G_q(y,x) \right| \ll \sup_{t \ge y}\frac{S(t;z)\log t}{t}
+ \int_y^\infty \frac{S(t;x)\log t}{t^2}\, dt
\ll \frac{1}{(\log z)^{C+1}}.
\]
Thus, the summand on the left is $\ge 1/(2z)$ for $O(z(\log z)^{-C-1})$ primes 
$q\in (z,2z]$.  Summing over dyadic intervals, we find that $\sup_{x>y} |G_q(y,x)|\ge 1/q$
for $O(\pi(u)/\log^C u)$ primes $q\le u$.
 For the other (non-exceptional) $q$, from Proposition
\ref{EKsump} and Theorem \ref{thm:S(q)}  we obtain
\[
 \EKq = 2\log (y-1) + O(1) - (q-1) \sum_{\substack{p\le y \\ p\equiv 1\pmod{q}}} \frac{\log p}{p-1},
\]
where $y \order q^2 (\log q)^{2B+1}$.  Finally, the Brun-Titchmarsh inequality and partial
summation gives 
\[
 \sum_{\substack{q^2 < p\le y \\ p\equiv 1\pmod{q}}} \frac{\log p}{p-1} \ll
\frac{\log (y/q^2)}{q} \ll_C \frac{\log\log q}{q}.
\]
This proves (i).  To obtain (ii), insert into \eqref{range} the bound
$E(t;q) \ll \sqrt{t} \log q$ valid under ERH (apply partial summation
to \cite[\S 20, (14)]{Da}), take
$y=q^2 (\log q)^{C+10}$ and argue as in part (i).
To prove (iii), substitute Conjecture EH for the Bombieri-Vinogradov Theorem
and take $y=z^{1+\eps}$ in the above argument.
\end{proof}

Part (ii) of Lemma \ref{xq2} may also be deduced from a general
bound for $\gamma_K^{\phantom{p}}$ due to Ihara \cite[Proposition 2]{I}.

\begin{lem}\label{admissible}
  For any $M>0$, there is an admissble set $\{a_1,\ldots,a_k\}$ with $\sum_i 1/a_i > M$. 
\end{lem}

\begin{proof}
 Let $p_1=3$ and, recursively for each $k\ge 2$, let $p_k$ be the smallest prime
for which $p_k\not\equiv 1\pmod{p_j}$ for all $j<k$.  Thus $p_2=5$, $p_3=17$, $p_4=23$, etc.
Erd\H os in \cite{Erd}, answering a question of S. Golomb, proved 
that $\sum_{k=1}^\infty 1/p_k$ diverges.  For a given $M$, let $J$ be so large that
if $\BB = \{ 2(p_j+1) : 1\le j\le J\}$, then  $\sum_{b\in \BB} 1/b > M$.
We now deduce that $\BB$ is admissible.
Let $F(n)=n\prod_{b\in \BB} (bn+1)$.  Observe that
by construction, if $r$ is prime and $r=p_j$ for some $j$, then
none of the elements of $\BB$ are congruent to $2\pmod{r}$.  Hence, if $4n\equiv -1\pmod{r}$, 
then $r\nmid F(n)$.  If $r$ is a prime and $r\ne p_j$ for every $j$, then
 none of the elements of $\BB$ are congruent to $1\pmod{r}$.  Consequently,
 if $2n\equiv -1\pmod{r}$, then $r\nmid F(n)$.  
\end{proof}

According to Granville \cite{Gr}, Lemma \ref{admissible} was conjectured by Erd\H os in 1988.
A proof is given in \cite[Theorem 3]{Gr}.
We showed above that Lemma \ref{admissible} is actually a 
simple corollary of Erd\H os' 1961 paper \cite{Erd}.

\begin{proof}[Proof of Theorem \ref{HL-EK} and Proposition \ref{2089}]
Let $M \ge 0$ be arbitrary.  Using Lemma \ref{admissible}, there is an admissible
 set $\{a_1,\ldots,a_k\}$ so that $\sum_i 1/a_i > M+2$.   By Lemma \ref{xq2} (i),
for all but $O(u/\log^{k+2} u)$ primes $q\le u$, 
\[
 \EKq = 2\log q +O_M(\log\log q) -  q \sum_{\substack{p\le q^2
\\ p\equiv 1\pmod{q}}} \frac{\log p}{p-1}.
\]
Assuming Conjecture HL, there are $\gg u/\log^{k+1} u$ primes $q\le u$ for which
$a_iq+1$ is prime for $1\le i\le k$.  For such primes $q > a_k+1$,
\[
 q \sum_{\substack{p\le q^2 \\ p\equiv 1\pmod{q}}} \frac{\log p}{p-1} \ge
\sum_{i=1}^k \frac{\log q}{a_i} > (M+2) \log q.
\]
Theorem \ref{HL-EK} follows.

Proposition \ref{2089} follows by taking $M=0$ in the above argument and noting that
we may take an admissible set with $k=2089$.
\end{proof}

\begin{proof}[Proof of Theorem  \ref{EH-EK}]
 Fix $\eta>0$.  Assuming Conjecture EH and using Lemma \ref{xq2} (iii),
we see that for all but $O(\pi(u)/\log^C u)$ primes $q\le u$,
\be\label{EHEK1}
 \EKq = (1+\eta^2) \log q + O_C(\log\log q) - q  
\sum_{\substack{p\le q^{1+\eta^2} \\ p\equiv 1\pmod{q}}} \frac{\log p}{p-1}.
\ee
On the other hand, by Lemma \ref{sieve} (with $\delta=\eta/2$ and $\eps=\eta^2$),
for all but $O(\eta \pi(u))$ primes $p\le u$, the above sum on $p$ is $\le (\eta\log q)/q$.
Hence, taking $C=1$, for all but $O(\eta \pi(u))$ primes $p\le u$,
$(1-\eta)\log q \le \EKq \le (1+\eta)\log q$ for large enough $q$.  
As $\eta$ is arbitrary, part (i) follows.

To show part (ii) concerning limit points of $\EKq/\log q$, 
start with \eqref{EHEK1} and let $\eps=\eta^2$.
Let $\AA=\{a_1,\ldots,a_k\}$ be an admissible set and let $m(\AA)=\sum_i 1/a_i$.
Assuming Conjecture HL, there are $\gg u/\log^{k+1} u$ primes $q\le u$ such that
$a_iq+1$ is prime for $1\le i\le k$.   By sieve methods \cite[Theorem 5.7]{HR}, 
the number of primes $q\le u$ for which
$a_iq+1$ is prime $(1\le i\le k)$ and $bq+1$ is also prime is $O(\frac{b}{\phi(b)}u/\log^{k+2} u)$, where
the implied constant depends on $\AA$.
Summing over even $b\le q^{\eps}$, $b\in \AA$, we find that there are
$O(\eps u/\log^{k+1} u)$ primes $q\le u$ with  $bq+1$ prime for some $b\le q^{\eps}$,
$b\not\in \AA$.  If $\eps$ is small enough, depending on $\AA$, then there are $\gg u/\log^{k+1}$
primes $q\le u$ for which $qa_i+1$ is prime ($1\le i\le k$) and $qb+1$ is composite
for all $b\le q^{\eps}$ such that $b\not\in \AA$.  For such $q$, \eqref{EHEK1} with $C=k+2$ implies that
\[
 \EKq = (1+\eps-m(\AA))\log q + O_k(\log\log q).
\]
As $\eps$ is arbitrary, we see that $1-m(\AA)$ is a limit point of $\{ \EKq/\log q: q\text{ prime}\}$.
Finally, it follows immediately from Lemma \ref{admissible} that $\{m(\AA): \AA\text{ admissible}\}$
is dense in $[0,\infty)$.  Indeed, given any $x>0$ and $\delta>0$, there is an admissible set of
integers $>1/\delta$ with $m(\AA)>x$.  As any subset of an admissible set is also admissible,
there is a subset $\AA'$ of $\AA$ with $|m(\AA')-x|<\delta$.
\end{proof}

%
%
\section{Upper bounds for $S(q)$}\label{Sec:S(q)}
%
%

We will give explicit upper bounds in Theorem \ref{thm:S(q)} for $S(q)$, making  use of
explicit estimates for prime numbers from \cite{RS}.  Note that
$f_p\ge 2$ implies that $q|(p^{f_p}-1)/(p-1)$, that is, 
\be\label{np}
\frac{p^{f_p}-1}{p-1} = q n_p, \quad n_p \ge 1.
\ee

\begin{lem}
\label{moeizaam}
For $x\ge 2$,
\[
\log x - 0.605 \le \sum_{p\le x}\frac{\log p}{p-1}\le \begin{cases}
\log x - 0.142 & (x\ge 9) \\
\log x - \tfrac12  & (x\ge 467.4). \end{cases}
\]
Also, 
\[
 \sum_{p \ge x}\frac{\log p}{p^3-1}\le \frac{0.6}{x^2} \quad (x > 2).
\]
\end{lem}

\begin{proof}
For the first estimate, we note that
$$
\sum_{p\le x}\frac{\log p}{p-1} =  \sum_{p\le x}\frac{\log
  p}{p}+\sum_{p\le x} \frac{\log p}{p(p-1)}.
$$
The latter sum can be easily bounded by $0.756$. 
The first estimate then is derived on invoking
\cite[Theorems 6, 21]{RS} to deal with $x\ge 1000$ and direct numerical calculation
for smaller $x$.
For $x\ge 7481$ one has $0.98x\le \sum_{p\le x}\log p\le 1.01624x$, as was
shown by Rosser and Schoenfeld \cite[Theorems 9 and 10]{RS}. 
{From} this one easily infers that for $x\ge 7481$
$$\sum_{p>x}\frac{\log p}{p^3-1}\le \frac{x}{x^3-1} 
\Big(-0.98+1.01624\pfrac32 \frac{x^3}{(x^3-1)}\Big).$$
For $k=2$, the right side is $\le 1.0525 x^{-1}$ and for $k=2$, the right side is
$\le 0.545 x^{-2}$.  For $x<7481$, we explicitly calculate the sum 
using
\[
 \sum_{p > x}\frac{\log p}{p^3-1} = -\frac{\zeta'(3)}{\zeta(3)} - 
\sum_{p\le x} \frac{\log p}{p^3-1}. \qedhere
\]
\end{proof}


\subsection{A simple upper bound}


\begin{lem}
\label{sommieb}
Let $q$ be a prime with $q \ge 5$. We have
$$S(q)\le \frac{\log q+1}{2q}.$$
\end{lem}

\begin{proof}
First, if $f_p=2$, then  $p=2kq-1$ for a positive integer
$k$.  As $p\ge 13$, we have
$p^2-1\ge 6(p+1)^2/7$.  Thus,
\begin{align*}
\label{222}
\sum_{p\equiv -1 \pmod{q}} \frac{\log p}{p^2-1} &\le \frac{7}{6}
\sum_{k=1}^\infty \frac{\log (2kq)}{4k^2 q^2} \\ &=
\frac{(7/6) \( \zeta(2) \log(2q) - \zeta'(2) \)}{4q^2} \le
\frac{0.48\log q + 0.61}{q^2}.
\end{align*}
Next, suppose $p>q$ and $f_p\ge 3$.  Combining the latter estimate and Lemma \ref{moeizaam},
we conclude that
\be\label{p>q}
 \sum_{p>q,f_p\ge 2} \frac{\log p}{p^{f_p}-1} \le \frac{0.48\log q+1.21}{q^2}.
\ee

Now suppose $p<q$ (so that $f_p\ge 3$).  If $q\ge 83$, by Lemma \ref{moeizaam}
\[
S'(q) =  \sum_{p<q} \frac{\log p}{p^{f_p}-1} \le \frac{1}{q}\sum_{p<\sqrt{q}}
\frac{\log p}{p-1} + \sum_{p>\sqrt{q}} \frac{\log p}{p^3-1}\le
\frac{0.5\log q+0.458}{q}.
\]
On combining this estimate with \eqref{p>q} yields the claimed bound for $q\ge 83$.
For $5\le q< 83$,
direct calculation shows that $S'(q) \le \frac{\log q-0.5}{2q}$ and the claimed
bound on $S(q)$ follows from \eqref{p>q}.
\end{proof}

Lemma \ref{sommieb} is strong enough in order to prove 
Theorem \ref{maintwo}.  However, with a refined analysis, we can obtain a sharper inequality
when $q$ is large.


\subsection{Refined upper bound}\label{fluca}


Note that in case $q$ is a Mersenne prime we have
$$S(q)\ge \frac{\log 2}{2^{f_2}-1}=\frac{\log 2}{q}.$$ Actually, the only
$q$ we have been able to find for which $S(q)>(\log 2)/q$
are the Mersenne primes. It thus is conceivable that if $q$ is not a 
Mersenne prime, then always $S(q)< (\log 2)/q$. For a given $\epsilon>0$
it also appeared to us that the primes $q$ for which $S(q)>\epsilon/q$ have density zero. In what follows, we prove that this is the case.
In general $S(q)$ is relatively large if $q$ almost equals a number
of the form $p^r-1$ with $p$ small. For example, if $2q=3^r-1$ for some $r$ (e.g. 
when $r=3,7,13,71$),
then $S(q)>(\log 3)/(2q)$. The above remarks show that the upper bound in 
the first part of Theorem \ref{thm:S(q)}, except
for the constant, is likely optimal.

\begin{proof}[Proof of Theorem \ref{thm:S(q)}]
We prove both (a) and (b) simultaneously. 
If $5 \le  q\le 10^{30}$, Lemma \ref{sommieb} gives $S(q)<35.1/q$ and (a) follows.  
Now suppose $q>10^{30}$.
We first consider three ranges for $p$:
\begin{enumerate}
 \item[(i)] $p>q$,
\item[(ii)] $p<q$ and $f_p\le F = \lceil \frac{\log q}{3\log\log q} \rceil$,
\item[(iii)] $f_p \ge F+1$ and $p>\log^4 q$.
\end{enumerate}
 Inequality \eqref{p>q} gives a good bound for the contribution of the primes in 
the range (i) to $S(q)$.
Note that
given $f\ge 3$, there are at most $f-1$ primes $p<q$ with $f_p=f$. 
By \eqref{np}, $q \le 2p^{f-1}$, hence the contribution to $S(q)$ from a given $f$ is
\[
 \le \frac{(f-1) \log [ (q/2)^{1/(f-1)} ]}{(q/2)^{f/(f-1)}-1} \le 
2.83 \frac{\log q}{q^{1 + \frac{1}{f-1}}}.
\]
If $f\le F$, then $q^{\frac{1}{f-1}} \ge \log^3 q$ and the
contribution to $S(q)$ from such $f$ is 
\be\label{caseii}
 \le \frac{2.83 (F-2)}{q\log^2 q} \le \frac{2.83}{3q(\log q)\log\log q}.
\ee
For $p$ counted in the range (iii), $p^{f_p-1} \ge p^F \ge q^{4/3}$.  By Lemma \ref{moeizaam},
the contribution
to $S(q)$ from such $p$ is
\be\label{caseiii}
\le \frac{1}{q^{4/3}} \sum_{\log^4 q < p < q} \frac{\log p}{p-1} \le \frac{\log q}
{q^{4/3}}.
\ee
By \eqref{p>q}, \eqref{caseii} and \eqref{caseiii}, the contribution to
$S(q)$ from $p$ in ranges (i)--(iii) is
\be\label{casei-iii}
O\pfrac{1}{q\log q} \qquad \text{ and also } \qquad \le \frac{1}{310 q}.
\ee

The primes $p$ not considered in ranges (i)--(iii) satisfy $p\le \log^4 q$ and
$f_p > F$.
We now take a brief interlude to prove (b).  The contribution to $S(q)$ from those $p$
with $f_p \ge F' = \lceil \frac{2\log q}{\log 2} \rceil$ is $\le 2\sum_p (\log p)p^{-F'}
=O(q^{-2})$.  As $f_p|(q-1)$, we have dealt with all ranges unless $q-1$ has a divisor in
$(F,F')$.  But this is rare; specifically, by Theorems
1 and 6 of \cite{Hxyz}, the number of $q\in (x,2x]$ with such a divisor is 
$O(\pi(x) (\log\log\log x/\log\log x)^{-0.086})$.  By \eqref{casei-iii}, (b) follows.

Next, we continue proving (a), by considering further ranges:
\begin{enumerate}
 \item[(iv)] $p\le e^{41}$,
\item[(v)] $e^{41} < p \le \log^4 q$ and $n_p \ge \min(p,f_p)$,
\item[(vi)] $e^{41} < p \le \log^4 q$  and $n_p < \min(p,f_p)$.
\end{enumerate}
Trivially, by Lemma \ref{moeizaam}, the contribution to $S(q)$ in case (iv) is
\be\label{caseiv}
\le \frac{1}{q} \sum_{p\le e^{41}} \frac{\log p}{p-1} \le \frac{40.5}{q}.
\ee
For ranges (v) and (vi), observe that $\log q \ge e^{41/4}$.  Since 
$f_p \ge \frac{\log q}{\log p}$, the contribution to $S(q)$ in case (v) is
\be\label{casev}
\begin{split}
&\le \frac{1}{q\log q} \sum_{e^{41}<p\le \log^4 q} \frac{\log^2 p}{p-1} +
\sum_{p>e^{41}} \frac{\log p}{qp(p-1)} \\
&\le \frac{4\log\log q (4\log\log q - 40.895)}{q\log q} + \frac{10^{-10}}{q} 
\le \frac{1}{416q}.
\end{split}
\ee
Here we used again Lemma \ref{moeizaam}, together with the fact that the maximum of
$x(x-b)e^{-x/4}$ occurs at $x=(b+8+\sqrt{b^2+64})/2$ (here $x=4\log\log q$).

Now consider range (vi). We will show that $f_p$ is prime. Indeed, assume that $f_p$ is composite.
Then 
$$
\frac{p^{f_p}-1}{p-1}=\prod_{\substack{d\mid f_p\\ d>1}} \Phi_d(p),
$$
where $\Phi_d(X)\in \Z[X]$ is the $d$th cyclotomic polynomial.
There exists some divisor $d_0>1$ of $f_p$ such that $q\mid \Phi_{d_0}(p)$ (in fact
$d_0=f_p$, but this is not needed for the proof). Hence,
$$
n_p\ge \prod_{\substack{d\mid f_p\\ d\ne 1,d_0}}\Phi_d(p).
$$
Since $f_p$ is not prime, the number $f_p$ has at least three divisors. Let $d_1>1$ be any divisor of $f_p$ different from $d_0$. Then
$$
n_p\ge \Phi_{d_1}(p)>(p-1)^{\phi(d_1)}\ge p-1,
$$
so $n_p\ge p$, a contradiction. Hence, $f=f_p$ is a prime factor of $q-1$. 
By Fermat's Little Theorem, 
$p^f\equiv p \pmod{f}$. Further, if $p\equiv 1({\rm mod~}f)$, then 
$(p^f-1)/(p-1)$ is a multiple of $f$. Otherwise, $p-1$ is invertible modulo $f$, and since $p^f-1\equiv p-1({\rm mod~}f)$, we get that $(p^f-1)/(p-1)$ is congruent to $1$ modulo $f$. Hence,
$$
qn_p=\frac{p^f-1}{p-1}\equiv 0,1 \pmod{f},
$$
and since $q\equiv 1({\rm mod~}f)$, we conclude that $n_p\equiv 0,1({\rm mod~}f)$.
But $n_p<f$, hence $n_p=1$ and
\be\label{pfq}
 \frac{p^f-1}{p-1} = q.
\ee
On writing the left hand side as $\sum_{j=0}^{f-1}p^j$, we that
in particular, $p|(q-1)$.  Since $q-1$ has at most $\frac{\log q}{\log\log q}$ prime
factors $>\log q$, the contribution to $S(q)$ from $p\in (\log q,\log^4 q]$ is
\be\label{casevi-1}
 \le \frac{\log q}{q\log\log q} \cdot \frac{4\log\log q}{\log q-1} \le \frac{4.004}{q}.
\ee
Let $\mathcal{P}$ be the set of primes satisfying \eqref{pfq} which are
in the interval $(e^{41},\log q]$.
We cover the interval in dyadic intervals of the form ${\mathcal I}_k=[2^k,2^{k+1})$ 
with $2^k\le \log q$,
and we look at ${\mathcal P}_k={\mathcal P}\cap {\mathcal I}_k$.
We will show  below that ${\mathcal P}_k$ has 
at most one element, and hence
\[
\frac{1}{q} \sum_{p\in\mathcal{P}} \frac{\log p}{p} \le 
\frac{1}{q} \sum_{k\ge 59} \frac{k\log 2}{2^k-1} \le \frac{1}{10^{15} q}.
\]
Combined with \eqref{casei-iii}, \eqref{caseiv}, \eqref{casev} and \eqref{casevi-1}, this
proves the theorem.

Now assume that ${\mathcal P}_k$ has at least two elements for some $k$, so that $k\ge 59$.
Let $p_1<p_2$ be any two elements in ${\mathcal P}_k$ with
\[
 q = \frac{p_1^{f_1}-1}{p_1-1} = \frac{p_2^{f_2}-1}{p_2-1}.
\]
Since the function $f\mapsto (p^f-1)/(p-1)$ is increasing for all fixed $p$, it follows that 
$f_1>f_2$. Now 
\begin{equation}
\label{eq:p1p2}
(p_2-1)p_1^{f_1}-(p_1-1)p_2^{f_2}=p_2-p_1.
\end{equation}
Thus, 
\begin{equation}
\label{eq:b1}
\left|\frac{(p_1-1)}{(p_2-1)} p_2^{f_2} p_1^{-f_1}-1\right|=\frac{p_2-p_1}{(p_2-1)p_1^{f_1}}< \frac{1}{p_1^{f_1}}\le \frac{1}{2^{kf_1}}.
\end{equation}
On the left, we use a lower bound for a linear form in three logarithms. Note that 
since $p_2>p_1$ this expression is not zero.
Now all three rational numbers $(p_1-1)/(p_2-1),~p_1$ and $p_2$ have height $<2^{k+1}$. Thus, Matveev's bound from \cite{Mat} (see also Theorem 9.4 in \cite{BLM}) tells us at once that
\begin{eqnarray}
\label{eq:b2}
\log \left|\frac{(p_1-1)}{(p_2-1)} p_2^{f_2} p_1^{-f_1}-1\right| & > & -1.4\times 30^{6}\times 3^{4.5}(1+\log(4f_1)) \left(\log(2^{k+1})\right)^3\nonumber\\
& > & -4.77\times 10^{10} (k+1)^3
(1+\log(4f_1)).
\end{eqnarray}
Thus, comparing bounds \eqref{eq:b1} and \eqref{eq:b2}, we get that
$$
k f_1\log 2 < 4.77\times 10^{10} (k+1)^3(1+\log(4f_1)).
$$
Since $k\ge 59$,
\[
f_1 <  \frac{4.77\times 10^{10}}{\log 2}\left(\frac{k+1}{k}\right) (k+1)^2 (1+\log(4f_1))
<  7\times 10^{10} (k+1)^2\log(4f_1).
\]
Here, we used the fact that $\log(4f_1)\ge \log(4F) > 37$, so
 $1+\log(4f_1)<\frac{38}{37}\log(4f_1)$. This gives
$$
4f_1< 2.876 \times 10^{11}(k+1)^2 \log(4f_1).
$$
For $A>10^{12}$, the inequality $x<A\log x$ implies that $x< \frac98 A\log A$ and hence
$$
f_1< 8.1 \times 10^{10} (k+1)^2 \( 26.4 + 2\log(k+1) \).
$$
Since $f_1\log p_1>\log q$, $\log p_1 < (k+1)\log 2$ and $2^{k}\le \log q$, we have
\[
2^k \le \log q <  (\log 2) \times 10^{11} (k+1)^3 (26.6+2\log(k+1)).
\]
This implies $k\le 58$, a contradiction.
\end{proof}


\section{Proof of theorems \ref{e0e1} and \ref{maintwo}}\label{sec:thm12}


Let 
\[ 
E_q(t)=\Psi(t;q,1)-\frac{t}{q-1},{\rm ~where~}\Psi(t;q,1)=\sum_{\substack{n\le x \\ n\equiv 1\pmod{q}}} \Lambda(n).
\]
 Let $R=9.645908801$.  We say that $\beta_0$ is an \emph{exceptional zero}
for a prime $q$ 
if $\beta_0 \ge 1 - 1/(R\log q)$ and $L(\beta_0,\chi)=0$, where $\chi$ 
is the quadratic character
modulo $q$.  Let $B(q)=1$ if $\beta_0$ exists, and $B(q)=0$ otherwise.

\begin{lem}\label{Mc}
 Suppose $q\ge 10000$ is prime.  Then, for $x\ge e^{R\log^2 q}$,
\[
 |E_q(x)| \le \frac{1.012 x^{\beta_0}}{q} B(q) +
\frac89 x \sqrt{\frac{\log x}{R}} \exp\left\{
- \sqrt\frac{\log x}{R} \right\}.
\]
\end{lem}

The proof of Lemma \ref{Mc} comes from estimates in McCurley \cite{Mc},
and will be given later in Section \ref{sec:Mc}.

\begin{proof}[Proof of Theorem \ref{e0e1}]
Propositions \ref{TqL} and \ref{EKsump} imply that
\be\label{e0e1-1}
 (q-1)\frac{e_1(q)}{e_0(q)} = 1 -\gamma - \frac{2\log q}{q^2-1} - S(q) + 
\lim_{x\to\infty} \Biggl[ \frac{\log x}{q-1}-\sum_{\substack{n\le x \\ n\equiv 1\pmod{n}}}
\frac{\Lambda(n)}{n} \Biggr].
\ee
By partial summation, for any $y>2q$ we have
\be\label{e0e1-2}
 \lim_{x\to \infty}\Big( \sum_{\substack{y<n\le x \\ n\equiv 1\pmod{q}}}
\frac{\Lambda(n)}{n} - \frac{\log(x/y)}{q-1} \Big)= -\frac{E_q(y)}{y} +
\int_y^\infty \frac{E_q(t)}{t^2}\, dt.
\ee
By Lemma \ref{Mc},
\be\label{EBH}
\Big| \int_y^\infty \frac{E_q(t)}{t^2}\, dt - \frac{E_q(y)}{y} \Big| \le B(q)
\frac{1.012(2-\beta_0)y^{\beta_0-1}}{(1-\beta_0)q} + 
\frac89 \pfrac{2RW^2+(4R+1)W+4R }{e^W},
\ee
where $W=\sqrt{\frac{\log y}R }$.

Taking $y=\exp(4 R\log^2 q)$ (so that $W=2\log q$), we obtain
\[
  \left|(q-1)\frac{e_1(q)}{e_0(q)} - (1 -\gamma)\right| \ll \frac{\log y}{q} + \frac{B(q)}{1-\beta_0}
+  \sum_{\substack{n \le y \\ n\equiv 1\pmod{q}}} \frac{\Lambda(n)}{n}. 
\]
By Proposition \ref{sumsmallp} and Theorem \ref{thm:S(q)}, the above sum on $n$ is
\[
 \le S(q) + \frac{2\log y + 2(\log q)\log\log(y/q)}{q-1} \ll \frac{\log^2 q}{q}.
\]
The first three parts of Theorem \ref{e0e1} now follow: for the first part, use
Lemma \ref{beta}; for the second part use Siegel's theorem \cite[\S 21]{Da} which states
that for every $\eps>0$, $\beta_0 \ge 1 - 
C(\eps) q^{-\eps}$ for an (ineffective) constant $C(\eps)$; for the third part,
we assume $\beta_0$ doesn't exist.

Finally, on ERH we have $E_q(t) \ll t^{1/2}\log^2 t$, uniformly in $q\le t$ 
\cite[\S 20, (14)]{Da}.  Hence, if $y\ge q$ then
\[
\Big|\int_y^\infty \frac{E_q(t)}{t^2}\, dt - \frac{E_q(y)}{y}\Big| \ll \frac{\log^2 y}{y^{1/2}}.
\]
Taking $y=q^3$ in the above argument 
yields $\EKq=O((\log q)(\log \log q))$ and hence the final estimate in Theorem \ref{e0e1}.
\end{proof}

{\bf Remarks}.
The estimate  $\EKq = O((\log q)\log\log q)$, valid under ERH,
was proved independently by Badzyan \cite{Badz}. Note that a third way to establish
it is by using \cite[Proposition 2]{I}.
 Unconditionally, Ihara
et al. \cite{IKMS} have shown that $\EKq\ll_\eps q^{\epsilon}$
(implicit in the third estimate in Theorem \ref{e0e1}).
In a more recent paper \cite{KM}, Kumar Murty proved that $|\EKq|$ is $O(\log q)$ on average:
$$\sum_{Q/2<q\le Q}|\EKq|\ll (\pi(Q)-\pi(Q/2))\log Q.$$

\begin{proof}[Proof of Theorem \ref{maintwo}]
By \eqref{e0e1-1}--\eqref{EBH} (ignoring the summands in \eqref{e0e1-1} with $n\le y$), 
together with the exceptional zero estimate in Lemma \ref{beta}, 
we have for $q\ge 10000$ the estimate
\[
 (q-1)\frac{e_1(q)}{e_0(q)} \le 1-\gamma + \frac{\log y}{q-1}+ 
1.015 \frac{y^{-D/(q^{1/2}\log^2 q)}\log^2 q}{D q^{1/2}}
+  \frac89 \pfrac{2RW^2+(4R+1)W+4R }{e^W},
\]
where $D=3.125 \max(2\pi,\frac12\log q)$.
When $q\ge 30000$, we take $y=e^{1.44 R \log^2 q}$, so that $W=1.2\log q$ and 
$D\ge 16.1$.  A short calculation reveals that  $e_1(q)/e_0(q) < \frac12$. 

For $q<30000$ we use the results of explicit calculation of $\EKq$
(e.g., Table 1 and Lemma \ref{EKq30000}).  
\end{proof}

\begin{center}
\afterpage{\clearpage}
\begin{table}
\caption{Approximate values of $S(q)$, $\EKq$ and $e_1(q)/e_0(q)$.}
\medskip
\medskip
\begin{tabular}{|r|r|r||r|r||r|}\hline
$q$&$S(q)$&$qS(q)$&$\EKq$\quad&$\EKq/\log q$& $\ds (q-1)\frac{e_1(q)}{e_0(q)}$ \\ \hline\hline
$3$&$0.351646$&$1.054940$&$0.945497$&$0.860628$&$1.247179$\\
$5$&$0.077777$&$0.388887$&$1.720624$&$1.069083$&$0.897187$\\
$7$&$0.122829$&$0.859805$&$2.087594$&$1.072811$&$0.866519$\\
$11$&$0.009100$&$0.100103$&$2.415425$&$1.007310$&$0.657441$\\
$13$&$0.046201$&$0.600623$&$2.610757$&$1.017859$&$0.673826$\\
$17$&$0.004437$&$0.075432$&$3.581976$&$1.264280$&$0.642487$\\
$19$&$0.011009$&$0.209173$&$4.790409$&$1.626934$&$0.692657$\\
$23$&$0.000829$&$0.019080$&$2.611289$&$0.832815$&$0.536910$\\
$29$&$0.000347$&$0.010088$&$3.093731$&$0.918758$&$0.529900$\\
$31$&$0.036585$&$1.134139$&$4.314442$&$1.256394$&$0.599845$\\
$37$&$0.000929$&$0.034387$&$4.304938$&$1.192200$&$0.540802$\\
$41$&$0.000449$&$0.018445$&$3.971521$&$1.069461$&$0.520422$\\
$43$&$0.000218$&$0.009397$&$4.378627$&$1.164157$&$0.525317$\\
$47$&$0.000129$&$0.006083$&$4.799394$&$1.246548$&$0.525580$\\
$53$&$0.000214$&$0.011346$&$4.337736$&$1.092548$&$0.505056$\\
$59$&$0.000065$&$0.003863$&$5.433516$&$1.332548$&$0.515399$\\
$61$&$0.001438$&$0.087727$&$5.071085$&$1.233578$&$0.507672$\\
$67$&$0.000268$&$0.018017$&$5.292139$&$1.258626$&$0.502328$\\
\hline
$71$&$0.000612$&$0.043471$&$5.255258$&$1.232853$&$0.497650$\\
$73$&$0.001374$&$0.100374$&$4.066949$&$0.947905$&$0.479861$\\
$79$&$0.000496$&$0.039250$&$4.998276$&$1.143914$&$0.486679$\\
$83$&$0.000073$&$0.006119$&$3.033136$&$0.686409$&$0.459221$\\
$89$&$0.000349$&$0.031120$&$4.164090$&$0.927696$&$0.469899$\\
$97$&$0.000171$&$0.016587$&$4.891240$&$1.069191$&$0.473429$\\
$101$&$0.000012$&$0.001283$&$5.297012$&$1.147751$&$0.475323$\\
$103$&$0.000032$&$0.003301$&$5.144339$&$1.109954$&$0.472822$\\
$107$&$0.000030$&$0.003234$&$5.458274$&$1.168087$&$0.473907$\\
$109$&$0.000025$&$0.002756$&$6.906638$&$1.472207$&$0.486372$\\
$113$&$0.000024$&$0.002809$&$4.021730$&$0.850729$&$0.458353$\\
$127$&$0.005911$&$0.750763$&$5.088599$&$1.050454$&$0.468785$\\
$131$&$0.000029$&$0.003827$&$2.836826$&$0.581889$&$0.444355$\\
$137$&$0.000034$&$0.004791$&$4.937000$&$1.003459$&$0.458862$\\
$139$&$0.000079$&$0.011060$&$5.889168$&$1.193474$&$0.465287$\\
$149$&$0.000008$&$0.001234$&$5.983424$&$1.195741$&$0.462998$\\
\hline
\end{tabular}
\end{table}
\end{center}

\section{Proof of Lemma \ref{Mc}}\label{sec:Mc}


In \cite{Mc}, McCurley gives estimates for $E_q(x)$ under the assumption that
the exceptional zero $\beta_0$ doesn't exist.   It is simple to modify the
arguments to handle the case when $\beta_0$ does exist.  Define
\[
 L=\log q, \quad X =\sqrt\frac{\log x}{R}, \quad x=e^{\lambda RL^2}, \quad \lambda
=(1+\a)^2, \quad H=q^\a.
\]
In particular,
\be\label{XL}
X = (1+\a)L = \log(qH).
\ee
Also, since $q\ge 10000$, we have $x \ge 10^{355}$.  We take $\eta=\frac12$ in
\cite[Theorem 2.1]{Mc}, which gives
\[
 \Big| N(T,\chi)-\frac{T}{\pi}\log\pfrac{qT}{2\pi e} \Big| \le C_1\log(qT)+C_2,
\]
where $C_1=0.9185$, $c_2=5.512$ and $N(T,\chi)$ is the number of zeros of $L(s,\chi)$
with imaginary part in $[-T,T]$ and real part in $(0,1)$.  
Lemma 3.5 of \cite{Mc} concerns bounds for $\sum_{\chi\ne\chi_0} |b(\chi)|$ (where $b(\chi)$ is
the constant term in the Laurent expansion of $\frac{L'}{L}(s,\chi)$ about $s=0$)
and it is assumed that $\beta_0$ 
doesn't exist.  However, by \cite[(3.16)]{Mc}, the existence of $\beta_0$ 
contributes an extra amount $\le \frac{1}{14}q^{1/2}\log^2 q$ to the sum.
The estimate in this lemma is thus increased by an amount $\le 0.06$ if $\beta_0$
exists.  

We apply \cite[Theorem 3.6]{Mc} with $m=2$ and $\del=2/H \le 0.0002$.  In the notation
of this theorem,
\be\label{A2delta}
A_2(\delta)=\delta^{-2} \( 1 + 2(1+\del)^3+(1+2\del)^3\) \le 4.003 \del^{-2}.
\ee
Denote by $\rho=\beta+i\gamma$ a generic zero of a non-principal $L$-function
with $0<\beta<1$.
Then we have
\be\label{Eqmain}
\frac{q-1}{x}|E_q(x)| < (1+\del) \sum_{\chi\ne\chi_0} \sum_{\rho:|\gamma|\le H}
\frac{x^{\beta-1}}{|\rho|} + \frac{4.003}{\del^2} 
\sum_{\chi\ne\chi_0} \sum_{\rho:|\gamma|> H}\frac{x^{\beta-1}}{|\rho(\rho+1)(\rho+2)|}
+\del+\eps_1,
\ee
where, using the modified Lemma 3.5 of \cite{Mc},
\be\label{eps1}
\eps_1 < \frac{q}{x} \( \frac{\log q \log x}{\log 2}+\frac{q\log q}{4}+
15\log^2 q+56\log q + 12\) < 10^{-300} X e^{-X}.
\ee
To estimate the sums over $\rho$, let
\[
 R(T)=C_1\log (qT)+C_2, \qquad \phi_n(t) = t^{-n-1} 
\exp\left \{-\frac{\log x}{R\log(qt)}\right\}.
\]
By \cite[Lemma 3.7]{Mc}, for each $\chi\ne\chi_0$,
\be\label{sumrho1}
\sum_{\substack{\rho:|\gamma|\le H \\ \rho \ne \beta_0}} 
\frac{x^{\beta-1}}{|\rho|} < \eps_2 +\eps_3 + \eps_4,
\ee
where, by \eqref{XL},
\begin{align*}
 \eps_2 &=\frac{1}{2\sqrt{x}}\(\frac{\lambda L^2}{\pi}+\frac{2+\a}{\pi}L+\frac{R(H)}{H}
+2R(1)+C_1\) + \frac{qL+\a L^2}{x} < 10^{-100} X e^{-X}, \\
\eps_3 &= \phi_0(H) R(H) = \frac{C_1 X + C_2}{H} e^{-X} < 0.00016 X e^{-X}, \\
\eps_4 &= \frac12 \int_1^H \phi_0(t) \log\pfrac{qt}{2\pi}\, dt <
\frac12 \int_1^H \phi_0(t) \log (qt)\, dt \\
&=\frac{\log^2 x}{2R^2}\int_{(1+\a)L}^{(1+\a)^2 L} \frac{e^{-u}}{u^3}\, du <
\frac{\log^2 x}{2R^2(1+\a)^3L^3} \int_{(1+\a)L}^\infty e^{-u}\, du = \frac{Xe^{-X}}{2}.
\end{align*}
Therefore,
\be\label{eps234}
\eps_2+\eps_3+\eps_4 < 0.5002 X e^{-X}.
\ee
For each $\chi\ne\chi_0$,  \cite[Lemma 3.8]{Mc} implies that
\be\label{sumrho2}
 \sum_{\rho:|\gamma|> H}\frac{x^{\beta-1}}{|\rho(\rho+1)(\rho+2)|} < \eps_5+\eps_6+\eps_7,
\ee
where
\begin{align*}
 \eps_5 &= \frac{1}{2H^2\sqrt{x}}\( \frac{H}{2\pi}(1+\a)L+2R(H)+\frac{C_1}{3}\)+ \frac{4L}
{xH^2} < 10^{-100} \frac{Xe^{-X}}{H^2}, \\
\eps_7 &= R(H) \phi_2(H) = \frac{C_1 X + C_2}{H^3}e^{-X} < 0.00016\frac{Xe^{-X}}{H^2}, \\
\eps_6 &= \frac12\int_H^\infty C_1 \phi_3(t)+\phi_2(t)\log\pfrac{qt}{2\pi}\, dt <
\frac12 \int_H^\infty \phi_2(t) \log (qt)\, dt \\
&= \frac{q^2 \lam L^2}{4} \int_{\sqrt{2}}^\infty u e^{-\frac{X}{\sqrt{2}}(u+\frac1{u})}\, du
= \frac{q^2 \lam L^2}{2\pi} K_2(2\sqrt{2} X, \sqrt{2}),
\end{align*}
where $K_2$ is the incomplete Bessel function.  By \cite[Lemmas 4 and 5]{RS75},
\[
 K_2(z,x) \le \(x + \frac{2}{z}\) \pfrac{x^2}{z(x^2-1)} e^{-\frac{z}{2}(x+1/x)} \qquad
(x>1,z>0),
\]
hence
\[
 \eps_6 \le \frac{q^2}{2\pi} \( X + \frac{1}{2} \) e^{-3X} = \frac{X}{2\pi} \(1+ \frac{1}{2X}\) 
\frac{e^{-X}}{H^2} \le \frac{0.1678 Xe^{-X}}{H^2}.
\]
Therefore,
\be\label{eps567}
\eps_5+\eps_6+\eps_7 < \frac{0.168 X e^{-X}}{H^2}.
\ee
By \eqref{A2delta}, 
\[
 \frac{\del}{q-1} \le \frac{2.0003}{qH} = 2.0003 e^{-X} < \frac{2.0003}{L} Xe^{-X}.
\]
Combining this with estimates \eqref{Eqmain}, \eqref{eps1}, \eqref{sumrho1}, \eqref{eps234},
 \eqref{sumrho2} and \eqref{eps567}, we conclude that
\begin{align*}
 |E_q(x)| &< B(q)\frac{(1+\del)x^{\beta_0}}{(q-1)\beta_0} + Xe^{-X} x \Bigg[(1+\del)(0.5002)
+10^{-300}+0.168 \frac{A_2(\del)}{H^2} \Bigg] + \frac{\del x}{q-1} \\
&< B(q) \frac{1.012 x^{\beta_0}}{q} + \frac89 x X e^{-X}. \qed
\end{align*}


\noindent {\tt Acknowledgement}. K.~F. was supported by National Science Foundation
grants DMS-0555367 and DMS-0901339. F.~L. worked on this project during a visit to CWI in Amsterdam in October 2010 as a Beeger lecturer. Both F.  L. and P.~M. thank H. te Riele for his hospitality at CWI when part of this work was done.  During the
preparation of this paper, F. L. was also supported in part by Grants
SEP-CONACyT 79685 and PAPIIT 100508. P.~M. likes to also thank Y. Hashimoto and Y. Ihara for kindly sending him
\cite{veel}, respectively \cite{I2}, and helpful e-mail
correspondence.  V. Kumar Murty pointed out the relevance of \cite{IKMS} and \cite{KM}. MPIM-intern
Philipp Weiss computed $i_0$ and $i_1$.
The authors thank C. Pomerance for bringing to their attention
Erd\H os' paper \cite{Erd}, thank A. Granville for pointing out his paper \cite{Gr} and thank 
W. Bosma and  P. Pollack for helpful conversations.

\vfil\eject

\medskip\noindent Kevin Ford\par\noindent
{\footnotesize {Department of Mathematics}, {University of Illinois at Urbana-Champaign} \hfil\break
 {1409 West Green Street, Urbana, IL 61801, USA}\hfil\break
e-mail: {\tt ford@math.uiuc.edu} }

\medskip\noindent Florian Luca \par\noindent
{\footnotesize {Centro de Ciencias Matem{\'a}ticas},{ Universidad Nacional Autonoma de M{\'e}xico} \hfil\break
{C.P. 58089, Morelia, Michoac{\'a}n, M{\'e}xico} \hfil\break
e-mail: {\tt fluca@matmor.unam.mx}}

\medskip\noindent Pieter Moree \par\noindent
{\footnotesize Max-Planck-Institut f\"ur Mathematik,
Vivatsgasse 7, D-53111 Bonn, Germany.\hfil\break
e-mail: {\tt moree@mpim-bonn.mpg.de}}

\end{document}